\documentclass[twoside,web]{ieeecolor}
\let\labelindent\relax

\usepackage[utf8]{inputenc}
\usepackage{generic}
\usepackage{cite}
 
\usepackage[dvipsnames]{xcolor}  
\usepackage[colorlinks=true, linkcolor=blue, filecolor=blue, urlcolor=blue,citecolor=blue]{hyperref}

\usepackage{amsmath,amssymb,amsthm,mathtools,bm}
\usepackage{upgreek,fixmath,commath}
\usepackage{times,mathptm}
\usepackage{enumitem,float,graphicx}
\usepackage{caption,subcaption}
\usepackage{multirow, tabularx, caption}
\theoremstyle{theorem}
\newtheorem{theorem}{Theorem}[section]

\newtheorem{lemma}{Lemma}[section]
\newtheorem{proposition}{Proposition}[section]
\theoremstyle{definition}
\newtheorem{definition}{Definition}[section]
\newtheorem{assumption}{Assumption}[section]
\theoremstyle{remark}
\newtheorem{remark}{Remark}[section]
\newtheorem{example}{Example}[section]

\newcommand{\optim}[2]{\underset{#1}{\mathrm{#2}}}%

\newcommand{\R}{\mathbb{R}}

\newcommand{\K}{\mathsf{K}}
\newcommand{\U}{\mathsf{U}}
\newcommand{\V}{\mathsf{V}}
\newcommand{\X}{\mathsf{X}}
\newcommand{\bu}{\mathsf{u}}
\newcommand{\bv}{\mathsf{v}}
\newcommand{\bw}{\mathsf{w}}

\newcommand{\col}[1]{\mathrm{col}(#1)}     
\newcommand{\row}[1]{\mathrm{row}(#1)}

\DeclareMathOperator*{\argmin}{arg\,min}

\newcommand{\tb}[1]{\textcolor{blue}{#1}}

\def\BibTeX{{\rm B\kern-.05em{\sc i\kern-.025em b}\kern-.08em
		T\kern-.1667em\lower.7ex\hbox{E}\kern-.125emX}}
\begin{document}
	\title{Feedback Stackelberg-Nash equilibria in difference games with quasi-hierarchical interactions and inequality constraints}
  \author{Partha Sarathi Mohapatra, Puduru Viswanadha Reddy and Georges Zaccour
		\thanks{P. S. Mohapatra and P. V. Reddy are with the Department of Electrical Engineering, Indian Institute of Technology-Madras, Chennai, 600036, India.  Georges Zaccour is with GERAD and HEC Montr\'eal, Montr\'eal, Canada.
		{(e-mail:   ps\textunderscore mohapatra@outlook.com, vishwa@ee.iitm.ac.in, georges.zaccour@gerad.ca)}}} 
	\date{\today }
	\maketitle
	
	\begin{abstract}
In this paper, we study a class of two-player deterministic finite-horizon difference games with coupled inequality constraints, where each player has two types of decision variables: one involving sequential interactions and the other simultaneous interactions. We refer to this class of games as quasi-hierarchical dynamic games and define a solution concept called the feedback Stackelberg-Nash (FSN) equilibrium.  		
		Under separability assumption on cost functions, we  provide a recursive formulation of the FSN solution  using dynamic programming. We show that the FSN solution  can be derived from the parametric feedback Stackelberg solution of an associated unconstrained game involving only sequential interactions, with a specific choice of the parameters that satisfy certain  implicit complementarity conditions.  For the linear-quadratic case, we show that an FSN solution is obtained by reformulating these complementarity conditions as a single large-scale linear complementarity problem. Finally, we illustrate our results using a dynamic duopoly game with production constraints.  		
	\end{abstract}
	\section{Introduction}
	\label{sec:introduction}

Dynamic game theory (DGT) provides a mathematical framework for modeling multi-agent decision processes that evolve over time. Unlike static games, where players act once, dynamic games involve sequential or simultaneous decisions over a fixed or endogenously determined horizon. This temporal structure introduces strategic complexities, requiring players to anticipate and respond to future decisions. The two primary solution concepts in non-cooperative dynamic games are Nash and Stackelberg equilibria, distinguished by the information available to players. In a Nash game \cite{Nash:51}, information is imperfect; that is, each player makes a strategic choice without knowing the others’ decisions. In a Stackelberg game \cite{Stackelberg:52}, information is perfect; that is, each player knows the opponent's last move when choosing their decision. Put differently, while play is simultaneous in a Nash game, it is sequential or hierarchical in a two-player Stackelberg game, where one player acts first (the leader) and the other (the follower) best responds. Anticipating the follower’s reaction, the leader selects a strategy to optimize her performance index. DGT has been successfully applied in engineering, management science, and economics, where dynamic multi-agent decision problems naturally arise (see \cite{isaacs:65, Basar:99, Engwerda:05, Dockner:00, Haurie:12, Basar:08} and the handbook \cite{Basar:18}). In particular, Nash equilibrium has been used in communication networks \cite{Zazo:16}, collision avoidance \cite{Mylvaganam:17}, and formation control \cite{Gu:07}, while Stackelberg equilibrium has been applied to demand response in smart grids \cite{Maharjan:13}, supply chains \cite{He:07, Chutani:12, Jorgensen:01}, and traffic routing \cite{Groot:17}. 

Two key features frequently appear in dynamic game models, both in discrete and continuous time. First, most models neglect constraints linking control (strategies, decisions) and state variables. For instance, existence and uniqueness theorems for Nash and Stackelberg equilibria are typically stated without considering coupled decision sets or mixed control-state constraints; see, e.g., \cite{Basar:99, Dockner:00, Engwerda:05, Haurie:12}. However, real-world multi-agent decision problems involve constraints such as saturation limits, bandwidth restrictions, production capacities, budget constraints, and emission limits. Incorporating these factors into a dynamic game introduces equality and inequality constraints that couple players' decision sets at each stage. Prior work has examined Nash equilibria under such constraints. \cite{Reddy:15, Reddy:17} analyzed open-loop and feedback Nash equilibria for linear-quadratic (LQ) games with private-type constraints. More recently, \cite{Partha:23b} characterized open-loop Nash equilibria under coupled constraints, while \cite{Partha:23a} studied open-loop Nash equilibria in mean-field-type games with deterministic coupled constraints. To the best of our knowledge, hierarchical counterparts to these setups remain unexplored.

Second, most two-player games involve either simultaneous or sequential play, with few exceptions \cite{Breton:06, Basar:10, Bensoussan:13, Bensoussan:19, Huang:24, Xie:21}. In \cite{Breton:06}, players alternate as leader and follower when setting advertising budgets, influencing market shares and profits. In mixed-leadership games \cite{Basar:10, Bensoussan:13, Bensoussan:19, Huang:24, Xie:21}, each player controls two variables: they first make simultaneous decisions on one set, then react, again simultaneously, by selecting the other. This results in two Nash games played sequentially. Mixed-leadership games have been applied to cooperative advertising \cite{Bensoussan:19} and innovation and pricing decisions \cite{Huang:24}.

In this paper, we study a class of two-player finite-horizon dynamic games where, at any period \( k \), each player \( i \) has two types of decision variables, \( u_{k}^{i} \) and \( v_{k}^{i} \). Except in the final period, the decision process is decomposed into three stages: (i) the leader announces \( u_{k}^{1} \); (ii) the follower responds with \( u_{k}^{2} \); and (iii) both simultaneously choose \( v_{k}^{1} \) and \( v_{k}^{2} \). At the terminal period \( K \), interaction occurs only through simultaneous choices of \( v_{K}^{1} \) and \( v_{K}^{2} \). This structure defines what we term \emph{quasi-hierarchical} dynamic games, which differ from mixed-leadership games by combining sequential and simultaneous interactions rather than two simultaneous ones. We illustrate this framework with two generic examples. 

As a first example, we consider a supply chain with a manufacturer and a retailer, where the decision process in each period is decomposed into three stages. First, the manufacturer (leader) sets the wholesale price. Next, the retailer (follower) responds by choosing the consumer price. Finally, both simultaneously invest in demand-enhancing activities, such as national brand advertising by the manufacturer and local advertising by the retailer. Notably, advertising expenditures are typically constrained by the available budget. 

In the previous example, the players engage in a \emph{vertical} strategic interaction, where the manufacturer sells a product to the retailer, who then sells it to consumers. As a second example, we consider a foreign firm and a local company competing in the same market with two partially substitutable products, where demand for each depends on both competitors' prices. Here, the strategic interaction is \emph{horizontal} .\footnote{The meaning and impact of cooperation differ in these two interactions. While coordination in a vertical channel is socially desirable, leading to higher profits and consumer surplus, coordination between competing firms amounts to collusion, harming consumer surplus and total welfare.}  
As before, the decision process in each period is decomposed into three stages. In the first two, firms sequentially adjust production capacities by adding or decommissioning equipment. In the third, they compete on price. The foreign firm first announces its global investment strategy, including in the market of interest, allowing the local firm to observe its capacity adjustments before making its own. Thus, investment decisions are sequential, with the foreign firm as leader and the local firm as follower, while pricing decisions are made simultaneously. Here, production quantity must be non-negative and constrained by capacity.  

Inspired by these examples, our dual objective is to characterize equilibrium strategies for two-player quasi-hierarchical dynamic games with mixed coupled inequality constraints and to provide a method for computing them.   
\subsection{Related literature} 
Nash equilibrium was first introduced in dynamic games in \cite{Starra:69, Starrb:69}.  Similarly, \cite{Simman:73} extended the Stackelberg solution to multi-period settings using a control-theoretic framework. In \cite{Simman3:73}, the authors introduced the feedback Stackelberg solution, where the leader enforces stage-wise policies rather than a global one.   When the leader has dynamic information and announces a policy for the entire game \cite{Bensoussan:15}, computing the Stackelberg solution, though well-defined, becomes challenging due to infinite-dimensional reaction sets in players' optimization problems. Indirect methods using an incentive approach were proposed in \cite{Basar:79, Basar:80} to derive such global Stackelberg solutions in discrete and continuous time.  In multi-player Stackelberg games, players are grouped into leaders and followers. \cite{Simman2:73} analyzed interactions where players within each group played Nash among themselves. The single-leader, multi-follower case has been studied in LQ settings under disturbances \cite{Kebriaei:017} and in a mean-field framework \cite{Moon:18}. \cite{Li:20, Lin:23} proposed algorithms for computing solutions in this setting with nonlinear dynamics.  
In \cite{Basar:10}, the authors introduced mixed-leadership games under an open-loop information structure, and  specifically, for LQ differential games. The existence of a unique open-loop Stackelberg equilibrium for two-player LQ differential games with mixed leadership was examined in \cite{Bensoussan:13}. Stochastic mixed-leadership games under a feedback information structure were analyzed in \cite{Bensoussan:19}, assuming the diffusion term was independent of players’ controls. In \cite{Huang:24}, a mixed-leadership game with state and control dependent diffusion was studied under a feedback information structure.  

In the DGT framework, simultaneous interactions have been studied with dynamic mixed state-control constraints. The existence of constrained open-loop and feedback Nash equilibria for a specific class of LQ difference games with affine inequality constraints was analyzed in \cite{Reddy:15, Reddy:17}. These works considered two types of control variables: one influencing state evolution and another affecting only the constraints, ensuring the first type remained uncoupled. More recently, \cite{Partha:23b, Partha:23a} characterized open-loop Nash equilibria in games with coupled constraints and in mean-field-type LQ games, respectively.  In \cite{Mondal:19}, necessary and sufficient conditions for an open-loop Stackelberg solution in LQ difference games with constraints were studied, but only for cases where the leader had linear state-control inequality constraints. Global Stackelberg solutions for stochastic games under adapted open-loop and closed-loop memoryless information structures with convex control constraints were studied in \cite{Zhang:21}. In \cite{Xie:21}, mixed-leadership games with input constraints were studied, restricting constraints to leader-controlled variables while leaving follower variables unrestricted.
\subsection{Contributions} 
 \tb{The contribution of our paper to the dynamic games literature is threefold. First, we introduce the new class of quasi-hierarchical dynamic games, which we believe has practical relevance, as illustrated by the supply chain and duopoly examples. A key distinction from its closest class, mixed-leadership games \cite{Basar:10, Bensoussan:13, Bensoussan:19, Huang:24, Xie:21}, lies in the information structure and player interactions. Instead of two successive simultaneous interactions, we have one sequential interaction followed by a simultaneous one. In short, we propose a different model of strategic interactions, leading to a different solution.} 

\tb{Second, we define an equilibrium concept for this class, the feedback Stackelberg-Nash (FSN) equilibrium, and provide sufficient conditions for its existence.  
Despite the similar name, this solution concept is fundamentally different from the Stackelberg-Nash solution in multi-player Stackelberg games \cite{Simman2:73, Kebriaei:017, Moon:18, Li:20, Lin:23}. In the latter, players have a single type of decision variable, with fixed interactions—simultaneous within each group and sequential between leaders and followers. In contrast, quasi-hierarchical dynamic games involve two types of decision variables, with leaders and followers interacting sequentially in one and simultaneously in the other.}  

 \tb{Finally, we provide a computational approach for the FSN equilibrium. We show that a solution to the original constrained difference game can be obtained from a parametric feedback Stackelberg solution of an associated unconstrained parametric game with only sequential interactions, using a specific choice of parameters that satisfy implicit complementarity conditions. Further, we show that the FSN equilibrium of a linear-quadratic game can be obtained by reformulating these complementarity conditions as a single large-scale linear complementarity problem.}

 The rest of the paper is organized as follows. Section \ref{sec:preliminaries} introduces finite-horizon nonzero-sum difference games with coupled inequality constraints and quasi-hierarchical interactions. Section \ref{sec:CFS} defines the feedback Stackelberg-Nash (FSN) solution. Section \ref{sec:FSNValue} provides a sufficient condition for its existence. Section \ref{sec:FSNfrompFS} shows how the FSN solution can be derived from the parametric feedback Stackelberg solution of an associated unconstrained parametric game, with a specific parameter choice satisfying implicit complementarity conditions. Section \ref{sec:CLQDG} specializes these results to a linear-quadratic setting with affine inequality constraints. Section \ref{sec:Numerical} illustrates our results through a dynamic duopoly game between a foreign and a local firm. Finally, Section \ref{sec:Conclusions} presents our conclusions.
\vskip1ex       
\noindent 
	\textbf{Notation:}    
    We denote real numbers by $\mathbb{R}$, non-negative real numbers by $\mathbb{R}_{+}$, the $n$-dimensional Euclidean space by $\mathbb{R}^n$, and $n \times m$ real matrices by $\mathbb{R}^{n \times m}$. For any $A \in \mathbb{R}^{n \times m}$, its transpose is $A^\prime \in \mathbb{R}^{m \times n}$. The identity and zero matrices are denoted by $\mathbf{I}$ and $\mathbf{0}$, with dimensions inferred from context. Let $A\in \mathbb R^{n\times n}$ and $a\in \mathbb R^n$, where
    the index set $\{1,2,\cdots,n\}$ is partitioned into $K$ blocks of sizes $n_1,n_2,\cdots, n_K$, such that
    $n=n_1+n_2+\cdots+n_K$. The $(i,j)$-th block submatrix of $A$, corresponding to rows in block $i$ and columns in block $j$, is denoted by $[A]_{ij} \in \mathbb R^{n_i\times n_j}$, and the $i$-th block of $a$ is denoted by $[a]_i\in \mathbb R^{n_i}$. The column vector $[v_1^\prime, \dots, v_n^\prime]^\prime$ is written as $\text{col}(v_1, \dots, v_n)$ or, more compactly, $\text{col}(v_k)_{k=1}^n$. The block diagonal matrix with diagonal elements $M_1, M_2, \dots, M_K$ is denoted by $\oplus_{k=1}^{K}M_k$.      
Vectors $x, y \in \mathbb{R}^n$ are complementary if $x \geq 0$, $y \geq 0$, and $x^\prime y = 0$, denoted as $0 \leq x \perp y \geq 0$. The composition of functions $f: \mathbb{R}^l\rightarrow \mathbb{R}^m$ and $g:\mathbb{R}^n\rightarrow \mathbb{R}^l$ is   $(f \circ g):\mathbb{R}^n \rightarrow \mathbb{R}^m$.
\begin{table*}[h]
\caption{Summary of abbreviations and variables}
 \centering 
\def\arraystretch{1.15} 
\begin{tabular}{ l l l l } 
\hline
Type & Symbol & Definition &Equations \\
\hline
\multirow{5}{5em}{Abbreviations} &CNZDG &Two-player nonzero-sum finite-horizon difference games with inequality constraints &\eqref{eq:state}-\eqref{eq:objective}\\ 
&pNZDG &Unconstrained two-player parametric nonzero-sum difference game &\eqref{eq:pnzdg} \\ 
&FSN &Feedback Stackelberg-Nash solution &\eqref{eq:ConSE}, \eqref{eq:Nash3rdstage}-\eqref{eq:followerresponse1}\\ 
&pFS &Parametric feedback Stackelberg solution &\eqref{eq:pFS}, \eqref{eq:ParametricRR}-\eqref{eq:pfollower}, \eqref{eq:Parametricsolutions}\\
&pCP  / pLCP / LCP& Parametric  /   Parametric linear  /  Linear  complementarity problem & \eqref{eq:nLCP} / \eqref{eq:LQpLCP} / \eqref{eq:LQMmap}\\
&&&\\
\multirow{3}{5em}{Controls} & $u_k^i$ / $v_k^i$ & Sequential / simultaneous type control of player $i$ at time $k\in\K_l$ &\eqref{eq:state},\eqref{eq:LQstate} / \eqref{eq:constraints},\eqref{eq:LQconstraints}\\ 
& $\U_k$ / $\V_k$ &  Joint admissible action set for sequential / simultaneous type controls at time $k\in\K_l$ & \eqref{eq:Uadmissible} / \eqref{eq:Vadmissible}\\  
& $\U_k^i$ / $\V_k^i$ &  Admissible action set for sequential / simultaneous type controls of player $i$ at time $k\in\K_l$ &\eqref{eq:admissibleuset} / \eqref{eq:admissiblevset}\\ 
&&&\\
\multirow{6}{5em}{Strategies} & $\gamma_k^i$ & State-feedback sequential type strategy of player $i$ at time $k\in\K_l$ &\eqref{eq:nash1}-\eqref{eq:leaderannouncement}\\  
& $\Gamma_k^i$ & Set of admissible sequential type state-feedback strategies of player $i$ at time $k\in\K_l$ &\eqref{eq:nash1}-\eqref{eq:leaderannouncement}\\  
&$\psi_{k}^{i}$ & Joint strategy of player $i$ in CNZDG  at time $k\in\K$ &\eqref{eq:ConSE}\\  
&$\Uppsi_{k}^{i}$ & Set of admissible joint strategy of player $i$ in CNZDG at time $k\in\K$ &\eqref{eq:ConSE}\\ 
& $\xi_k^i$ & State-feedback strategy of player $i$ in pNZDG at time $k\in\K_l$ &\eqref{eq:pFS}\\  
& $\Xi_k^i$ & Set of admissible state-feedback strategies of player $i$ in pNZDG at time $k\in\K_l$ &\eqref{eq:pFS}\\
&&&\\
\multirow{4}{5em}{Miscellaneous} & $W^i$ / $W_p^i$ & Value functions of player $i$ in CNZDG / pNZDG & \eqref{eq:valfuncon} / \eqref{eq:valfunpar}\\ 
& $\mathsf{R}^2_k$ / $\bar{\mathsf{R}}^2_k$ & Optimal response set of the follower at time $k\in\K_l$ in CNZDG / pNZDG &\eqref{eq:followerresponse} / \eqref{eq:rationalreactionsetp}\\  
& $\X_k$ / $\bar{\X}_k$ & Reachable set at time $k\in\K_r$ in CNZDG / pNZDG &\\  
& $\mathsf{w}_{k}, \uptheta_{k}$ & Parameters at time $k\in\K$ in pNZDG &\eqref{eq:pnzdg} \\
\hline
\end{tabular}

\end{table*}
    
	\section{Preliminaries}
	\label{sec:preliminaries}
	\subsection{Dynamic game with inequality constraints}        
        In this section, we introduce a class of two-person discrete-time nonzero-sum finite-horizon dynamic games with inequality constraints (CNZDG). Let $\{1,2\}$ denote the players and $\mathsf{K}=\{0,1,\dots,K\}$ the set of decision instants. We define $\mathsf{K}_{l}:=\mathsf{K} \setminus \{K\}$ and $\mathsf{K}_{r}:=\mathsf{K} \setminus \{0\}$.   
       For any $k \in \mathsf{K}_l$, the \emph{downstream} decision instants are $\{k+1, \dots, K\}$, and for any $k \in \mathsf{K}_r$, the \emph{upstream} decision instants are $\{0, \dots, k-1\}$. 
         At each  instant $k\in \mathsf{K}_{l}$, player $i\in \{1,2\}$ selects an action \(u_{k}^{i} \in \mathbb{R}^{m_{i}}\), influencing the state variable \(x_{k} \in \mathbb{R}^{n}\), which evolves according to:  
        \begin{align}
        	x_{k+1} = f_k(x_k, u_k^1, u_k^2),\quad k \in \K_l,\label{eq:state}
        \end{align}
        where $f_{k}:\mathbb{R}^{n}\times \mathbb{R}^{m}\rightarrow \mathbb{R}^{n}$, with $m:=m_{1}+m_{2}$ and the initial state $x_{0} \in \mathbb{R}^{n}$ is given.  
                At each $k\in \mathsf{K}$, both players also endowed with decision variables $v_{k}^{i} \in \mathbb{R}^{s_{i}}, i\in \{1,2\}$, which do not affect the state dynamics but influence their decisions through the following mixed inequality constraints        
	\begin{align}
		h_k^i(x_k, v_k^{1}, v_k^{2}) \geq 0, ~ v_k^i\geq 0,~k \in \K,~i \in  \{1,2\},\label{eq:constraints}
	\end{align}
        where $h_{k}^{i}:\mathbb{R}^{n}\times \mathbb{R}^{s}\rightarrow \mathbb{R}_{+}^{c_{i}}$ with $s:=s_1+s_2$. We denote the joint actions of players by $\mathsf{u}_{k}:=\mathrm{col}(u_{k}^{1},u_{k}^{2}),~k\in \mathsf{K}_{l}$ and $\mathsf{v}_{k}:=\mathrm{col}(v_{k}^{1},v_{k}^{2}),~k\in \mathsf{K}$. The   profile of actions of player $i\in \{1,2\}$ is denoted by $(\tilde{\mathsf{u}}^{i},\tilde{\mathsf{v}}^{i})$, where $\tilde{\mathsf{u}}^{i}:=\mathrm{col}(u_{k}^{i})_{k=0}^{K-1}$ and $\tilde{\mathsf{v}}^{i}:=\mathrm{col}(v_{k}^{i})_{k=0}^{K}$. 
       Player $i $ minimizes the following stage-additive cost functional:
\begin{align}
  J_i(x_0, (\tilde{\bu}^1, \tilde{\bv}^1), (\tilde{\bu}^2, \tilde{\bv}^2))=
      g_{K}^i(x_K, \bv_K)+\sum_{k=0}^{K-1}g_k^i\left(x_k, \bu_k, \bv_k\right),\label{eq:objective}
\end{align}
      where $g_{k}^{i}:\mathbb{R}^{n}\times \mathbb{R}^{m}\times \mathbb{R}^{s}\rightarrow \mathbb{R}$, $k\in \mathsf{K}_{l}$ and $g_{K}^{i}:\mathbb{R}^{n}\times \mathbb{R}^{s}\rightarrow \mathbb{R}$ denote the instantaneous and terminal cost functions of player $i,$ respectively.
 We assume that the dynamics, constraints, and cost functions are common knowledge among the players.
Next, we define the admissible action spaces for both types of decision variables. The constraints \eqref{eq:constraints} are coupled, meaning that at each time instant $ k \in \mathsf{K}$, player 1's control action $ v_{k}^{1} $ depends on player 2's control action $v_{k}^{2}$ and vice versa. Given $x_k \in \mathbb{R}^n $, the joint admissible action set  $\mathsf{V}_k(x_k)$ at  $ k \in \mathsf{K} $ is defined as  
\begin{align} 
	\mathsf{V}_k(x_k) := \{(v_k^{1}, v_k^{2}) \in \mathbb{R}^{s} ~|~ \eqref{eq:constraints} \text{ holds}\}.\label{eq:Vadmissible}
\end{align} 
Similarly, for $ x_k \in \mathbb{R}^n $, the joint admissible action set $ \mathsf{U}_k(x_k) $ at $ k \in \mathsf{K}_l $ is given by  
\begin{align}
	\mathsf{U}_k(x_k) := \{(u_k^{1}, u_k^{2}) \in \mathbb{R}^{m} ~|~ \mathsf{V}_{k+1}(f_k(x_k, u_k^1, u_k^2)) \neq \emptyset \}.\label{eq:Uadmissible}
\end{align}
For $ k \in \mathsf{K}_r $, we define the reachable set $ \mathsf{X}_k \subset \mathbb{R}^n $ as the set of all state variables at time instant $ k \in \mathsf{K}_r $ that can be reached when players $ i \in \{1,2\} $ choose any sequence of admissible actions $ (u_{\tau}^{1}, u_{\tau}^{2}) \in \mathsf{U}_{\tau}(x_{\tau}) $ for $ \tau = 0, \dots, k-1 $.
Next, we define the admissible action sets for player $ i \in \{1,2\} $ with decision variables $ u_k^i \in \mathbb{R}^{m_i} $ and $ v_k^i \in \mathbb{R}^{s_i} $, given $ x_k \in \mathsf{X}_k $ and the other player’s action $ (u_k^j, v_k^j) $, where $ j \neq i $:  
\begin{subequations}
\begin{align}
	\U^i_k(x_k,u_k^j) &:= \{ u_k^i \in \mathbb{R}^{m_i} \mid (u_k^1, u_k^2) \in \U_k(x_k) \}, ~ k \in \K_l,
	\label{eq:admissibleuset} \\
	\V^i_k(x_k,v_k^j) &:= \{ v_k^i \in \mathbb{R}^{s_i} \mid (v_k^1, v_k^2) \in \V_k(x_k) \}, ~ k \in \K.
	\label{eq:admissiblevset}
\end{align}  
\label{eq:admissibleiset}%
\end{subequations}
For notational brevity, we drop the arguments in the players' admissible action sets \eqref{eq:admissibleiset} for the remainder of the paper. 

\begin{remark}\label{rem:AdmissibleUV}
The dynamic game in \eqref{eq:state}-\eqref{eq:objective} was first studied in \cite{Reddy:15, Reddy:17, Reddy:19} within a linear-quadratic framework. In this class of games, the decision variables $\mathsf{u}_{k}$ at time $k$ influence the state variable $x_{k+1}$ at time $k+1$, while the decision variables $\mathsf{v}_{k+1}$ at time $k+1$ are constrained by $x_{k+1}$ through \eqref{eq:constraints}. Thus, although $\mathsf{u}_k$ has no explicit constraints, the feasibility of \eqref{eq:constraints} at $k+1$ imposes  coupling   \eqref{eq:Uadmissible} through the state equation \eqref{eq:state}. 
\end{remark}
        We make the following assumptions related to the dynamic game \eqref{eq:state}-\eqref{eq:objective}:
	\begin{assumption}\label{ass:A1}
		\begin{enumerate}[label = (\roman*)]
			\item \label{itm:1} The admissible action sets $\U_k(x_k) \subset \R^{m}$ for $k \in \K_l$ are such that the   action sets $\V_k(x_k)$ for all $k \in \K$ are nonempty, convex, closed and bounded.
		    \item \label{itm:3} The functions $\{f_k,\, k\in \K_l,\, h_k^i,\, g_k^i,\, k \in \K,\, i \in \{1,2\}\}$ are continuously differentiable in their arguments.	
                \item \label{itm:2} The matrices $\big\{\frac{\partial h_k^i}{\partial v_k^i},\, k \in \K,\, i \in \{1,2\}\big\}$ are full rank for each  $\bv_k\in\V_k(x_k)$ and $x_k\in\X_k$.
		\end{enumerate}
	\end{assumption}
        \tb{ 
From Remark \ref{rem:AdmissibleUV}, Assumption \ref{ass:A1}.\ref{itm:1} is necessary to ensure the feasibility of the joint action sets $\V(x_k)$. Additionally, Assumption \ref{ass:A1}.\ref{itm:1} requires $\V(x_k)$, for all \( k \in \K \), to be convex, closed, and bounded to guarantee the existence of a solution (see Lemma \ref{lem:nLCP}). Assumption \ref{ass:A1}.\ref{itm:2} ensures that constraint qualification conditions hold; see also \cite{Bueno:19} for alternative formulations. While verifying Assumption \ref{ass:A1} is generally challenging, it is straightforward for the LQ case in Section \ref{sec:CLQDG} (see Assumption \ref{ass:CLQDG} and Remark \ref{rem:LQVerification}).  
}
	\subsection{Quasi-hierarchical interaction and information structure}
	\begin{figure}[h]
		\centering
\includegraphics[scale=0.75]{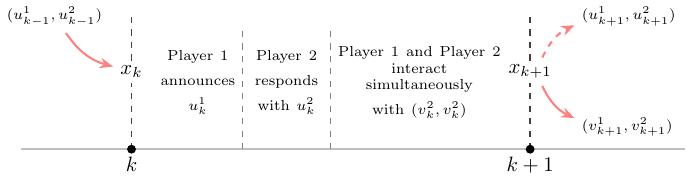}
		\caption{Three step decision process involving sequential and simultaneous interactions between the time instants $k$  and $k+1$.}
		\label{fig:threestages}
	\end{figure} 
   In this paper, we analyze a quasi-hierarchical interaction among players, as defined below:
\begin{definition}[Quasi-hierarchical interaction]\label{Def:A1_1} 
The leader and follower are denoted by the labels $1$ and $2,$ respectively. Between any two time instants $k$ and $k+1$, with $k\in \mathsf{K}_{l}$, the players interact in three stages. In the first stage, the leader announces her action $u_{k}^{1}$ and in the second stage, the follower responds by announcing her action $u_{k}^{2}$. In the third stage, the players choose simultaneously their decisions $v_{k}^{1}$ and $v_{k}^{2}$. At the terminal instant $K$, the players choose simultaneously the decision variables $v_{K}^{1}$ and $v_{K}^{2}$.
\end{definition}

The decision process between time instants \( k \) and \( k+1 \) is illustrated in Figure \ref{fig:threestages}. In the sequential interaction, the leader first announces her action \( u_k^1 \in \mathsf{U}_k^1 \), followed by the follower's response \( u_k^2 \in \mathsf{U}_k^2 \), which determines the next state \( x_{k+1} \). The state \( x_k \) influences the joint action set \( \mathsf{V}_k(x_k) \), from which the players simultaneously choose actions \( (v_k^1, v_k^2) \) in the third stage. Before the next time instant \( k+1 \), each player \( i \in \{1,2\} \) incurs a cost \( g_k^i(x_k, \mathsf{u}_k, \mathsf{v}_k) \). At the terminal time \( K \), both players interact simultaneously, and player \( i \) incurs a cost \( g_K^i(x_K, \mathsf{v}_K) \).
 
In the presence of constraints, \cite{Reddy:15,Reddy:17} introduce a constrained information structure in both open-loop and feedback forms, and demonstrate that this information structure leads to a semi-analytic characterization of Nash equilibrium strategies. Motivated by this aspect, in this paper, we assume a constrained feedback information structure, which is described as follows; see also \cite[Section III.A]{Reddy:17}.

\begin{definition}[Constrained feedback information structure]\label{def:CFSI} 
         The control action $(u_k^i,v_k^i)$ of player $i$ ($i\in \{1,2\}$) at stage $k$ is given by $u_k^i:=\gamma_k^i(x_k)\in \mathsf{U}_k^i$ and $v_k^i\in \mathsf{V}_k^i$, where $\gamma_k^i:\X_k\rightarrow \mathsf{U}_k^i$ is a measurable mapping and the set of all such mappings is denoted by $\Gamma_k^i$.
\end{definition}
	 
\section{Feedback Stackelberg-Nash solution }\label{sec:CFS}
       In this section, we present the feedback Stackelberg-Nash (FSN) solution for the dynamic game \eqref{eq:state}-\eqref{eq:objective} with quasi-hierarchical interaction. To this end, we introduce additional notation. Let $\psi _{k}^{i}:=(\gamma_{k}^{i},v_{k}^{i})\in \Uppsi_{k}^{i}$ denote the mapping associated with the joint action $(u_{k}^{i},v_{k}^{i}),\,\,i\in \{1,2\}$, where we represent the joint action space as $\Uppsi_{k}^{i}$, i.e., $\forall k\in  \mathsf{K}_{l},\Uppsi_{k}^{i}:=\Gamma _{k}^{i}\times \V_{k}^{i}$ and $\Uppsi_{K}^{i}:=\V_{K}^{i}$. Consequently, we have $(u_{k}^{i},v_{k}^{i})=(\gamma_{k}^{i}(x_{k}),v_{k}^{i})=\psi _{k}^{i}$ for $k\in \mathsf{K}_{l}$ and $\psi _{K}^{i}:=v_{K}^{i}$ at the terminal time. We denote the strategy of player $i$ by $\psi ^{i}:=((\gamma _{k}^{i}(x_{k}),v_{k}^{i})_{k\in \mathsf{K}_{l}},v_{K}^{i})$. The strategy space for player $i$ is denoted by $\Uppsi^{i}:=\prod_{k=0}^{K}\Uppsi_{k}^{i}$. For any $a,b\in \mathsf{K}$, with $0\leq a\leq b\leq K$, we denote the collection of actions for time periods from $a$ to $b$ (including both $a$ and $b$) by $\psi _{\lbrack a,b]}^{i}:=(\psi_{a}^{i},\cdots ,\psi _{b}^{i})$ and the corresponding admissible sets by $\Uppsi_{[a,b]}^{i}:=\prod_{k=a}^{b}\Uppsi_{k}^{i}$. We define the FSN solution as follows:


\begin{definition}[Feedback Stackelberg-Nash solution] \label{def:CFS} 
         A pair $(\psi ^{1\star },\,\psi ^{2\star })\in (\Uppsi^{1},\Uppsi^{2})$ constitutes a \emph{feedback Stackelberg-Nash solution} for CNZDG if the following conditions are satisfied: 
\begin{enumerate}
         \item For all $\psi _{\lbrack 0,K-1]}^{i}\in \Uppsi_{[0,K-1]}^{i}$,~$i\in \{1,2\}$, $(\psi _{K}^{1\star },\psi _{K}^{2\star })\in \mathsf{V}_{K}$ satisfy 
			\begin{subequations}\label{eq:ConSE}
				\begin{align}
					&\psi^{1\star}_{K}\in \argmin_{\psi^1_K \in \V_K^1(v_K^{2\star})}~ J_1\bigl(x_0, (\psi^1_{[0,K-1]},\psi^{1}_{K}), (\psi^{2}_{[0,K-1]}, \psi^{2\star}_{K})\bigr),\label{eq:Nash1K}\\
					&\psi^{2\star}_{K}\in \argmin_{\psi^2_K \in \V_K^2(v_K^{1\star})}~ J_2\bigl(x_0, (\psi^1_{[0,K-1]},\psi^{1\star}_{K}), (\psi^{2}_{[0,K-1]}, \psi^{2}_{K})\bigr).\label{eq:Nash2K}
				\end{align}
           \item Recursively for $k=K-1,\cdots ,0$ and for all $\psi _{[0,k-1]}^{i}\in \Uppsi_{[0,k-1]}^{i},~i\in \{1,2\}$, in the decision process between time instants $k$ and $k+1$,
	    \begin{enumerate}
            \item For all $\gamma _{k}^{i}\in \Gamma_{k}^{i}, i\in\{1,2\}$, $(v_{k}^{1\star},v_{k}^{2\star})\in \mathsf{V}_{k}$ satisfy in the third stage the following conditions:
				\begin{align} 
					&v_k^{1\star}\in  \argmin_{v_k^1 \in \V_k^1(v_k^{2\star})}~ J_1\bigl(x_0, \bigl(\psi^1_{[0,k-1]}, (\gamma^{1 }_k,v_k^{1 }), \psi^{1\star}_{[k+1,K]}\bigr),\notag\\
                    &\hspace{1.0in}\bigl(\psi^2_{[0,k-1]}, ( \gamma^{2}_k ,v_k^{2\star}), \psi^{2\star}_{[k+1,K]}\bigr)\bigr),
					\label{eq:nash1}\\
					&v_k^{2\star}\in  \argmin_{v_k^2 \in \V_k^2(v_k^{1\star})}~J_2\bigl(x_0, \bigl(\psi^1_{[0,k-1]}, (\gamma^{1 }_k,v_k^{1\star }), \psi^{1\star}_{[k+1,K]}\bigr),\notag\\ 
                    &\hspace{1.0in}\bigl(\psi^2_{[0,k-1]}, ( \gamma^{2}_k ,v_k^{2}), \psi^{2\star}_{[k+1,K]}\bigr)\bigr).
					\label{eq:nash2}
				\end{align} 
            \item For every $\gamma _{k}^{1}\in \Gamma _{k}^{1}$, there exists a unique map $\mathsf{R}_{k}^{2}:\Gamma _{k}^{1}\rightarrow \Gamma_{k}^{2}$ such that the following condition is satisfied in the second stage: 			
				\begin{align} 
				&\mathsf R_k^2(\gamma^{1}_k)=\argmin_{\gamma_k^2 \in \Gamma_k^2}~J_2\bigl(x_0, \bigl(\psi^1_{[0,k-1]}, (\gamma^{1 }_k,v_k^{1\star }), \psi^{1\star}_{[k+1,K]}\bigr),\notag\\ 
                &\hspace{1.0in}\bigl(\psi^2_{[0,k-1]}, ( \gamma^{2}_k ,v_k^{2\star}), \psi^{2\star}_{[k+1,K]}\bigr)\bigr).
				\label{eq:followerresponse}
			\end{align} 
            \item In the first stage, $\gamma _{k}^{1\star }\in \Gamma _{k}^{1}$ satisfies the condition 
				\begin{align} 
			&\gamma^{1 \star}_k\in \argmin_{\gamma_k^1 \in \Gamma_k^1}~ J_1\bigl(x_0, \bigl(\psi^1_{[0,k-1]}, (\gamma^{1 }_k,v_k^{1\star }), \psi^{1\star}_{[k+1,K]}\bigr),\notag\\ 
            &\hspace{0.75in}\bigl(\psi^2_{[0,k-1]}, ( \mathsf R_k^2(\gamma^{1}_k) ,v_k^{2\star}), \psi^{2\star}_{[k+1,K]}\bigr)\bigr),
			\label{eq:leaderannouncement}
		\end{align} 
            with $\gamma_k^{2\star}=\mathsf R_k^2(\gamma_k^{1\star})$.
				\end{enumerate}				 
			\end{subequations}%
		\end{enumerate}
 The feedback Stackelberg-Nash equilibrium costs incurred by the leader and the follower are given by $J_{1}(x_{0},\psi ^{1\star},\psi ^{2\star})$ and $J_{2}(x_{0},\psi ^{1\star},\psi ^{2\star})$, respectively.
\end{definition}
In Definition \ref{def:CFS}, conditions \eqref{eq:ConSE} characterize the FSN solution via dynamic programming, decomposing its computation into solving $ K+1 $ static games backward in time. At the terminal time $ K $, the only decision variables are $ (v_{K}^{1}, v_{K}^{2}) $, in which players interact simultaneously, yielding the Nash equilibrium $ (\psi _{K}^{1\star },\psi_{K}^{2\star }) $.
 More specifically, the conditions \eqref{eq:Nash1K}-\eqref{eq:Nash2K} imply that whatever admissible strategies used by the players to reach time $K$, that is, for any arbitrary choice of admissible strategies $\psi _{\lbrack 0,K-1]}^{i}\in \Uppsi_{[0,K-1]}^{i}$, the decisions specified by FSN solution $(\psi _{K}^{1\star },\psi_{K}^{2\star})$ satisfy conditions \eqref{eq:Nash1K}-\eqref{eq:Nash2K} at $K$. For $ k \in \mathsf{K}_{l} $, players interact sequentially via $ (u_{k}^{1}, u_{k}^{2}) $ and simultaneously via $ (v_{k}^{1}, v_{k}^{2}) $. The FSN solutions are computed for any arbitrary choice of admissible upstream strategies $ (\psi_{\lbrack 0,k-1]}^{1}, \psi_{\lbrack 0,k-1]}^{2}) $ by fixing downstream decisions at $ (\psi _{\lbrack k+1,K]}^{1\star}, \psi _{[k+1,K]}^{2\star}) $. Between time $ k $ and $ k+1 $, the interaction follows the three-stage process in Figure \ref{fig:threestages}.  Using dynamic programming, conditions \eqref{eq:nash1}-\eqref{eq:nash2} imply that stage three is simultaneous, with Nash outcome $ (v_{k}^{1\star }, v_{k}^{2\star }) $. At stage two, fixing these decisions, \eqref{eq:followerresponse} implies that the follower best responds as $ \mathsf{R}_{k}^{2}(\gamma _{k}^{1}) $ for any leader announcement $ \gamma _{k}^{1} \in \Gamma _{k}^{1} $. At stage one, \eqref{eq:leaderannouncement} implies the leader minimizes her cost, accounting for the follower’s best response, to obtain her FSN Stackelberg strategy $ \gamma _{k}^{1\star} \in \Gamma _{k}^{1} $.   
\begin{remark}
If the decision variables \((v_{k}^{1}, v_{k}^{2})\), the constraints \eqref{eq:constraints}, and the corresponding objective terms \eqref{eq:objective} are absent, the FSN solution reduces to the canonical feedback-Stackelberg solution; see \cite[Definition 3.29]{Basar:99}.
\end{remark}
\begin{remark}
In Definition \ref{def:CFS}, the follower's rational reaction set $\mathsf{R}_{k}^{2}(\gamma _{k}^{1})$ is needed to be a singleton for every leader's announcement $\gamma _{k}^{1}\in \Gamma _{k}^{1}$ at all time $k\in \mathsf{K}_{l}$. In the subgame starting at $k$, if $\mathsf{R}_{k}^{2}(\gamma _{k}^{1})$ is not a singleton, then the leader's cost depends on the strategy used by the follower from $\mathsf{R}_{k}^{2}(\gamma_{k}^{1})$. The leader can secure her cost against multiple follower's optimal responses from $\mathsf{R}_{k}^{2}(\gamma _{k}^{1})$ by considering the follower's response that gives her the worst cost. When $\mathsf{R}_{k}^{2}(\gamma _{k}^{1})$ is not a singleton, the leader cannot enforce her strategy on the follower, which creates an important difficulty in determining a Stackelberg equilibrium. This explains why the literature typically assumes that $\mathsf{R}_{k}^{2}(\gamma _{k}^{1})$ is a singleton \cite[Theorem 7.1 and Theorem 7.2]{Basar:99}.
\end{remark}
\begin{remark}\label{rem:multiNash}
In Definition \ref{def:CFS}, the recursive formulation of the FSN solution results in a static game at each time step \( k \), where downstream decisions are fixed at FSN solutions. If, at any \( k \), multiple solutions \( (v_{k}^{1\star}, v_{k}^{2\star}) \) exist, the upstream static games must be solved for each one of these solutions to obtain the full set of FSN solutions.
\end{remark}
%
 \section{Recursive formulation of feedback Stackelberg-Nash solution}\label{sec:ValueCFS}
In this section, we present a recursive formulation of the FSN solution. To this end, we assume the following structure for the players' cost functions:  

\begin{assumption}\label{ass:seperability}
	For all \( k \in \mathsf{K}_{l} \) and \( i \in \{1,2\} \), the cost functions are separable:  
	\begin{align}\label{eq:costseparability}  
		g_k^i(x_k, \bu_k, \bv_k) = g_{u, k}^i(x_k, \bu_k) + g_{v, k}^i(x_k, \bv_k).  
	\end{align}  
\end{assumption}  
This assumption rules out cross terms between sequential and simultaneous decision variables in the players' objectives. \tb{The following example illustrates the complications arising from such cross terms:}  

\begin{example}
         \tb{Consider a scalar two-period game with state dynamics given by $x_{1}=x_{0}-(u_{0}^{1}+u_{0}^{2})$ and the initial state $x_{0}\in \mathbb{R}$. The leader minimizes the cost function $J_{1}=4x_{1}x_{1}+2u_{0}^{1}u_{0}^{1}+u_{0}^{2}u_{0}^{2}+v_{0}^{1}v_{0}^{1}+v_{0}^{1}v_{0}^{2}$ subject to the constraints $x_{0}-v_{0}^{1}-2v_{0}^2\geq 0$, $v_{0}^{1}\geq 0$. Similarly, the follower minimizes the cost $J_{2}=x_{1}x_{1}+u_{0}^{2}u_{0}^{2}+u_{0}^{2}v_{0}^{2}+v_{0}^{2}v_{0}^{2}+v_{0}^{1}v_{0}^{2}$ subject to the constraints $x_{0}-v_{0}^{1}-2v_{0}^2\geq 0$ and $v_{0}^{2}\geq 0$.          
         For this example, Assumption \ref{ass:A1} holds. The joint action space \( \V_0(x_0) \) is nonempty for any \( x_0 > 0 \), convex, closed (as it is defined by affine inequalities), and bounded. Additionally, the full rank condition is satisfied since  $
         \frac{\partial h_0^1}{\partial v_0^1} = -1$ and $ \frac{\partial h_0^2}{\partial v_0^2} = -2$.}
         
         \tb{Since players have decision variables only for time instant \(0\), we consider a three-stage decision process for this period. Given any admissible sequential actions \(u_0^1\) and \(u_0^2\) for the leader and follower,   the Nash (simultaneous) solution at time   \(0\) is determined by the following optimization problems:
 	\begin{align*}
 		\min_{v_0^1\geq0}~&\big\{4x_{1}x_{1}+2u_{0}^{1}u_{0}^{1}+u_{0}^{2}u_{0}^{2}+v_{0}^{1}v_{0}^{1}+v_0^1v_0^{2\star}\big\}\\&\quad\textrm{subject to}~ x_{0}-v_{0}^{1}-2v_{0}^{2\star}\geq 0,\\
 		\min_{v_0^2\geq0}~~&\big\{x_{1}x_{1}+u_{0}^{2}u_{0}^{2}+u_{0}^{2}v_{0}^{2}+v_{0}^{2}v_{0}^{2}+v_0^{1\star}v_0^2\big\}\\&\quad\textrm{subject to}~ x_{0}-v_{0}^{1\star}-2v_{0}^2\geq 0.
 	\end{align*}
        Using the relation $x_{1}=x_{0}-(u_{0}^{1}+u_{0}^{2})$, the associated Lagrangians can be written as $4(x_{0}-(u_{0}^{1}+u_{0}^{2}))^{2}+2u_{0}^{1}u_{0}^{1}+u_{0}^{2}u_{0}^{2}+v_{0}^{1}v_{0}^{1}+v_{0}^{1}v_{0}^{2\star }-\mu _{0}^{1}(x_{0}-v_{0}^{1}-2v_{0}^{2\star})$ and $(x_{0}-(u_{0}^{1}+u_{0}^{2}))^{2}+u_{0}^{2}u_{0}^{2}+u_{0}^{2}v_{0}^{2}+v_{0}^{2}v_{0}^{2}+v_{0}^{1\star }v_{0}^{2}-\mu _{0}^{2}(x_{0}-v_{0}^{1\star}-2v_{0}^2)$. Here, $\mu _{0}^{1}$ and $\mu _{0}^{1}$ are the Lagrange multipliers associated with the coupled inequality constraints $x_{0}-v_{0}^{1}-2v_{0}^{2\star}\geq 0$ and $x_{0}-v_{0}^{1\star}-2v_{0}^2$, respectively. The associated KKT conditions are 
        \begin{equation}
\begin{rcases}
        		0\leq 2v_{0}^{1\star }+v_{0}^{2\star}+\mu _{0}^{1\star }\perp v_{0}^{1\star}\geq 0, \\
        		0\leq u_{0}^{2}+2v_{0}^{2\star }+v_{0}^{1\star}+2\mu_{0}^{2\star }\perp v_{0}^{2\star }\geq 0, \\
        		0\leq x_{0}-v_{0}^{1\star}-2v_{0}^{2\star}\perp \mu _{0}^{1\star }\geq 0,\\
                0\leq x_{0}-v_{0}^{1\star}-2v_{0}^{2\star}\perp \mu_{0}^{2\star }\geq 0.        		
        	\end{rcases} \label{eq:FollowerCon}
        \end{equation}
        Next, at stage two, the optimal response of the follower for any leader's admissible action $u_{0}^{1}$, upon fixing the third stage decisions at $(v_{0}^{1\star },v_{0}^{2\star })$, is obtained by solving the optimization problem \eqref{eq:followerresponse}, which after substituting for $x_{1}=x_{0}-(u_{0}^{1}+u_{0}^{2})$ is given by 
 	\begin{align*}
 		\optim{u_0^2}{min}~&\big\{(x_0-(u_0^1+u_0^2))^2+u_{0}^{2}u_{0}^{2}+u_0^2v_0^{2\star}+v_0^{2\star}v_0^{2\star}+v_0^{1\star}v_0^{2\star}\big\}\\&\quad\textrm{subject to}~~\eqref{eq:FollowerCon}.
 	\end{align*}
       The presence of the cross term $u_{0}^{2}v_{0}^{2}$ in the follower's objective results in a  mathematical programming problem with complementarity constraint (MPCC). In general, MPCC are extremely hard to solve because of the non-convexity and non-smoothness of the feasible set. Also, the lack of constraint qualification at every feasible point makes them intractable; see \cite{Holger:00, Ye:97}.}   
        \qed
\end{example}%
        Next, using Definition \ref{def:CFS}, we develop a recursive formulation of the FSN solution. To this end, we use the reachable set \(\mathsf{X}_{k} \subset \mathbb{R}^{n}\), the set of all state variables at time instant \(k\) that can be reached when players employ admissible strategies in $\Uppsi_{[0,k-1]}^{i},~ i\in\{1,2\}$. Let the FSN strategies of the players be denoted by \(\{\psi^{i\star} \equiv (\{(\gamma _{k}^{i\star}(x_{k}),v_{k}^{i\star })\}_{k\in \mathsf{K}_{l}},v_{K}^{i\star }),~i\in \{1,2\}\}\).  
        Recall that the FSN solution at \(k\) results in a static game at the same time instant, where downstream decisions from \(k+1\) to \(K\) are fixed at FSN strategies. Let \(G_{k}^{i}(x_{k},\psi _{k}^{1},\psi_{k}^{2})\) denote the cost incurred by player \(i\) in this game when players use $\psi _{k}^{i}\in \Uppsi_{k}^{i},~i\in \{1,2\}$  at time instant $k\in\K_l$. Then, we have  
 \begin{align}
 	G_k^i(x_k,\psi_k^1,\psi_k^2)&:=
 	g_{k}^i(x_k,(\gamma_k^1(x_k), \gamma_k^2(x_k)), (v_k^1,v_k^2))  \notag\\
 	& +\sum_{\tau=k+1}^{K-1} g_{\tau}^i(x_\tau,(\gamma_\tau^{1\star}(x_\tau), \gamma_\tau^{2\star}(x_\tau)), (v_\tau^{1\star}, v_\tau^{2\star}))\notag\\
 	& +g_{K}^i(x_K, (v_K^{1\star},v_K^{2\star})),\label{eq:playericost}
 \end{align}
        where $x_{k+1}=f_{k}(x_{k},\gamma _{k}^{1}(x_{k}),\gamma _{k}^{2}(x_{k}))$ and $x_{\tau +1}=f_{\tau }(x_{\tau },\gamma _{\tau }^{1\star }(x_{\tau}),\gamma_{\tau }^{2\star }(x_{\tau }))$ for $\tau =k+1,\cdots ,K-1$. Following the state-additive structure of the cost functions, ${G}_{k}^{i}$ can be written recursively backwards for $k=K-1,\cdots ,0$ as
 \begin{subequations}\label{eq:Grecur}
 	\begin{align}
 		&G_k^i(x_k,\psi_k^1,\psi_k^2)=
 		g_{k}^i(x_k,(\gamma_k^1(x_k), \gamma_k^2(x_k)), (v_k^1,v_k^2)) \notag\\
 	&\quad+G_{k+1}^i(f_k(x_k, \gamma_k^1(x_k), \gamma_k^{2}(x_k)) ,\psi_{k+1}^{1\star},\psi_{k+1}^{2\star}),\label{eq:Grecursive}\\
 	&G_K^i(x_k,\psi_K^{1\star},\psi_K^{2\star})= g_{K}^i(x_K, v^{1\star}_{K}, v^{2\star}_{K}).\label{eq:Grecursiveterm}
 	\end{align}
        Using the separability Assumption \ref{ass:seperability}, we introduce the auxiliary cost of player $i$ as
 \begin{align}
 	&\tilde{G}_{u, k}^i(x_k, \gamma_k^1(x_{k}),\gamma_k^2(x_{k})):=g_{u, k}^i(x_k, \gamma_k^1(x_{k}),\gamma_k^2(x_{k}))\notag\\
 	&\quad+G_{k+1}^i(f_k(x_k, \gamma_k^1(x_k), \gamma_k^2(x_k)), \psi^{1\star}_{k+1}, \psi^{2\star}_{k+1}).\label{eq:Gtilde}
 \end{align}
 \end{subequations} 
        Then, using \eqref{eq:Grecur}, player $i$'s cost \eqref{eq:playericost} is given by 
 \begin{align}
 	G_k^i(x_{k}, \psi^{1}_{k}, \psi^{2}_{k})
 	& = \tilde{G}_{u, k}^i(x_k, \gamma_k^1(x_{k}),\gamma_k^2(x_{k}))+ g_{v, k}^i(x_k, (v_k^1,v_k^2)).\label{eq:Gseparability}
 \end{align}
\subsection{Simultaneous interaction}
        The third-stage decision problem at any time instant $k\in \mathsf{K}_{l}$ given by \eqref{eq:nash1}-\eqref{eq:nash2} translates as
\begin{subequations} 
\begin{align} 
&v_k^{1\star} \in \argmin_{v_k^1\in \V_k^1(v_k^{2\star})}~ G_k^{1}(x_k,(\gamma_k^1(x_k),v_k^1),(\gamma_k^2(x_k),v_k^{2\star})), \label{eq:kstage31}\\
&v_k^{2\star} \in \argmin_{v_k^2\in \V_k^2(v_k^{1\star})}~G_k^{2}(x_k,(\gamma_k^1(x_k),v_k^{1\star}),(\gamma_k^2(x_k),v_k^{2})). \label{eq:kstage32}
\end{align} 
\end{subequations}
Clearly, from \eqref{eq:Gseparability}, the above third-stage problem results in a static parametric (in $x_{k}$) nonzero-sum game denoted by $\bm{\Uplambda}_{k}(x_{k}):=\bigl<\{1,2\},\mathsf{V}_{k}(x_{k}), \{g_{v,k}^{i}(x_{k},v_{k}^{1},v_{k}^{2})\}_{i\in \{1,2\}}\bigr>$, where each player $i\in \{1,2\}$ solves
\begin{align}
	&\min _{v_k^i\in \mathbb R^{s_i}} g_{v,k}^i(x_k,(v_k^i,v_k^{j\star})), \notag \\
	&\text{subject to}~ h_k^i(x_k,(v_k^i, v_k^{j\star}))\geq 0,~v_k^i\geq 0, \label{eq:subgame} 
\end{align}
where $j\in\{1,2\}$, $j\neq i$.
The KKT conditions for the third-stage interaction at time $k$ can be expressed as the following parametric (in $ x_k$) complementarity problem, denoted by $\mathrm{pCP}_{k}(x_{k})$:  
\begin{align} 
	&\mathrm{pCP}_k(x_k ):\quad  0 \leq\begin{bmatrix}
		\nabla{L}_k(x_k,  \bv_k^{\star},{\upmu}_k^{\star})\\
		h_k(x_k, \bv_k^{\star})
	\end{bmatrix} \perp\begin{bmatrix}
		\bv^{\star}_k\\
		{\upmu}_k^{\star}
	\end{bmatrix} \geq 0,\label{eq:nLCP} 
\end{align}
where  
\begin{align*}
	\nabla{L}_k(x_k,  \bv_k^{\star},{\upmu}_k^{\star})&= \col{\nabla_{v_k^i} g^i_{v,k}(x_k, \bv_k^{\star})-{\mu_k^{i\star}}^{\prime}\nabla_{v_k^i} h_k^i(x_k, \bv_k^{\star})}_{i=1}^{2},\\
	h_k(x_k, \bv_k^{\star})&=\col{h_k^i(x_k, \bv_k^{\star})}_{i=1}^{2},~ {\upmu}_k^{\star}=\col{\mu_k^{i\star}}_{i=1}^{2},
\end{align*}  
and $\nabla_{v_k^i} g^i_{v,k}(x_k, \bv_k^{\star})$ is the gradient of $g^i_{v,k}(x_k, \bv_k^{\star})$ with respect to $v_k^i$.
Here, $
L_k^i(x_k,v_k^i,\mu_k^i):=g_{v,k}^i(x_k,(v_k^i,v_k^{j\star}))-{\mu_k^i}^\prime  h_k^i(x_k,(v_k^i, v_k^{j\star}))$ is the Lagrangian associated with \eqref{eq:subgame} and 
  $\mu _{k}^{i }$ is the     multiplier for the constraint $ h_{k}^{i}(x_{k},(v_{k}^{i},v_{k}^{j\star }))\geq 0$ for $i,j\in\{1,2\}$, $i\neq j$.   

We denote the solution of \eqref{eq:nLCP}, if it exists, as $(\mathsf{v}^\star_k, \boldsymbol{\upmu}_k^\star) := \texttt{SOL}(\mathrm{pCP}_k(x_k))$. \tb{It is well known that a static game  can admit multiple Nash equilibria. Consequently, the upstream decision problems must be solved for each equilibrium; see Remark \ref{rem:multiNash}. This can be avoided if $\mathrm{pCP}_k(x_k)$ has a unique equilibrium at every time step $k \in \mathsf{K}$. To ensure this, we impose the following assumption on the cost functions $g_{v,k}^i$ in the third-stage game $\bm{\Uplambda}_k(x_k)$, based on \cite[Theorem 2]{Rosen:65}.}  

\begin{assumption}\label{ass:convex} 
      For every $(k,x_k)\in \mathsf{K }\times \mathsf{X}_k$ the cost functions $g^i_{v,k}(x_k,.):\mathbb{R}^s\rightarrow \mathbb{R}$,~$i\in \{1,2\}$ are \emph{diagonally strictly convex} (DSC). That is, for every $\boldsymbol{\alpha},\boldsymbol{\beta} \in \mathbb{R}^{s}$ we have 
	\begin{align} 
		(\boldsymbol{\alpha}-\boldsymbol{\beta})^\prime \nabla g_{v,k}(x_k,\boldsymbol{\beta}) + (\boldsymbol{\beta}-\boldsymbol{\alpha})^\prime\nabla g_{v,k}(x_k,\boldsymbol{\alpha})<0,
	\end{align}
where $\nabla g_{v,k}(x_{k},\mathsf{v}_{k}) = \col{\nabla_{v_k^i} g^i_{v,k}(x_k, \bv_k)}_{i=1}^{2}$. 
\begin{remark} 
	A sufficient condition for the cost functions to satisfy the DSC property, as given in \cite{Rosen:65}, is that the symmetric matrix $[G(x_{k},\mathsf{v}_{k}) + G^{\prime }(x_{k}, \mathsf{v}_{k})]$ be positive definite for all $\mathsf{v}_{k} \in \mathbb{R}^{s}$, where $G(x_{k},\mathsf{v}_{k})$ is the Jacobian of $ \nabla g_{v,k}(x_{k},\mathsf{v}_{k})$ with respect to $\mathsf{v}_{k}$.	
\end{remark}
\end{assumption}


\begin{lemma}\label{lem:nLCP}
      Consider the CNZDG described by \eqref{eq:state}-\eqref{eq:objective}. Let Assumptions \ref{ass:A1}, \ref{ass:seperability} and \ref{ass:convex} hold. Then, for every $(k,x_{k})\in \mathsf{K}\times \mathsf{X}_{k}$, $\texttt{SOL}(\mathrm{pCP}_{k}(x_{k}))$ is unique (a singleton), and the FSN control actions $\mathsf{v}_{k}^{\star }$ are obtained as $(\mathsf{v}_{k}^{\star },{\upmu}_{k}^{\star })=\texttt{SOL}(\mathrm{pCP}_{k}(x_{k}))$
\end{lemma}

\begin{proof} 
\tb{Under the DSC Assumption~\ref{ass:convex}, the individual cost functions
$
g^i_{v,K}(x_K,(\cdot,v_K^{j})):\mathbb{R}^{s_i}\rightarrow \mathbb{R}, \quad i,j\in\{1,2\},\ i\neq j,
$
are strictly convex for each $v_K^{j}\in\mathbb{R}^{s_j}$. By Assumption~\ref{ass:A1}.\ref{itm:1}, the constraint sets $\V_K(x_K)$ are nonempty, convex, closed, and bounded.  
Hence, the terminal-stage game $\bm{\Uplambda}_K(x_K)$ is convex, so a Nash equilibrium exists by \cite[Theorem~1]{Rosen:65}, and the DSC property ensures it is unique by \cite[Theorem~2]{Rosen:65}, given by
$
(\bv_K^{\star}, {\upmu}_K^{\star}) = \texttt{SOL}(\mathrm{pCP}_K(x_K)), \forall x_K \in \mathsf{X}_K.
$
For any $k \in \K_l$ and $x_k \in \mathsf{X}_k$, Assumption~\ref{ass:seperability} ensures that \eqref{eq:kstage31}–\eqref{eq:kstage32} define a static parametric game $\bm{\Uplambda}_k(x_k)$. The constraint sets $\V_k(x_k)$ are nonempty, convex, closed, and bounded (Assumption~\ref{ass:A1}.\ref{itm:1}) and the cost functions satisfy the DSC property (Assumption~\ref{ass:convex}).  
Following reasoning similar to the terminal stage, $\bm{\Uplambda}_k(x_k)$ is convex, a Nash equilibrium exists, and is unique. Hence, the FSN control actions $\mathsf{v}_k^{\star}$ obtained by solving the complementary problem \eqref{eq:nLCP} are unique for all $x_k \in \mathsf{X}_k$.}
\end{proof} 
\subsection{Sequential interaction}
       In the second stage, the follower solves the optimization problem \eqref{eq:followerresponse} for any leader's announcement $\gamma_{k}^{1}\in \Gamma _{k}^{1}$. Again, from Assumption \ref{ass:seperability} on separability of cost functions, the follower's optimal response is obtained as  
\begin{align}
	\mathsf R_k^2(\gamma_k^1)&=\argmin_{\gamma_k^2\in \Gamma_k^2}~G_k^2(x_k,(\gamma_k^1(x_k),v_k^{1\star}),(\gamma_k^2(x_k),v_k^{2\star}))\notag \\
	&=\argmin_{\gamma_k^2\in \Gamma_k^2}~\tilde{G}^2_{u,k}(x_k,\gamma_k^1(x_k),\gamma_k^2(x_k)).
\end{align}
       Next, in the first stage, the leader solves the optimization problem \eqref{eq:leaderannouncement} considering the follower's best response. That is, the leader's optimization problem is given by 
\begin{align}\label{eq:leadersproblem}
\gamma_k^{1\star}\in \argmin_{\gamma_k^1\in \Gamma_k^1}~G_k^1(x_k,(\gamma_k^1(x_k),v_k^{1\star}),((\mathsf R_k^2\circ \gamma_k^1) (x_k),v_k^{2\star})). 
\end{align} 
       Again, from Assumption \ref{ass:seperability} the leader's FSN solution $\gamma_{k}^{1\star }$ satisfies 
\begin{align}
	&\gamma_k^{1\star}\in \argmin_{\gamma_k^1\in \Gamma_k^1}~\tilde{G}_{u,k}^1(x_k,\gamma_k^{1}(x_k), (\mathsf R_k^2\circ \gamma_k^{1})(x_k)).\label{eq:Leadercond}
\end{align}
 
\begin{remark}
       In our model, Assumptions \ref{ass:seperability} (separability) and  \ref{ass:convex} (DSC property) enable to compute the third stage FSN decisions uniquely as $(\mathsf{v}_{k}^{\star},{\upmu}_{k}^{\star})=\texttt{SOL}(\mathrm{pCP}_{k}(x_{k}))$ at all time instants $k\in \mathsf{K}$. As a result, at time instant $k$, third stage decisions are decoupled from first and second stage decisions $(u_{k}^{1},u_{k}^{2})$. However, they are influenced by the sequential decisions $(u_{k-1}^{1},u_{k-1}^{2})$ taken at the upstream time instant $k-1$ indirectly through the state variable $x_{k}$.
\end{remark}


\begin{remark}
       In accordance with the definition of FSN solution, in \eqref{eq:leadersproblem}, it is assumed that the optimal response set of the follower is a singleton set. This condition is readily met when $\tilde{G}_{u,k}^2(x_k,\gamma_k^1(x_k),.)$ is a strictly convex function on $\mathsf{U}_k^2$; see also \cite[Theorem 7.1 and Theorem 7.2]{Basar:99}.
\end{remark}


\begin{remark}\label{rem:FSNodering}
       When the leader's optimization problem \eqref{eq:leadersproblem} at time $ k $ has multiple solutions $ \gamma _{k}^{1\star} $, she is indifferent among them as they yield the same stage-wise cost. However, the follower’s optimal response $ \gamma_{k}^{2\star} $ may vary, influencing the next-state $ x_{k+1} $ given by \eqref{eq:state}. Since $ x_{k+1} $ determines $ (v_{k+1}^{1\star},v_{k+1}^{2\star}) $, these multiplicities can result in different global FSN costs for the leader.  
       More precisely, if FSN solutions $ (\hat{\psi}^{1},\hat{\psi}^{2}) $ and $ (\bar{\psi}^{1},\bar{\psi}^{2}) $ satisfy $ J_{1}(x_{0},\hat{\psi}^{1},\hat{\psi}^{2}) < J_{1}(x_{0},\bar{\psi}^{1},\bar{\psi}^{2}) $, the leader prefers $ \hat{\psi}^{1} $ and enforces $ \hat{\gamma}_{k}^{1\star} $ at each instant. Given the uniqueness of simultaneous decisions (under Assumption \ref{ass:convex}) and the follower’s response, her cost depends entirely on $ \gamma_{k}^{1\star} $, establishing a total ordering of FSN solutions.  
       Thus, an FSN solution $ (\hat{\psi}^{1}, \hat{\psi}^{2}) $ is \emph{admissible} (see also \cite[Definition 3.30]{Basar:99}) if no other solution $ (\bar{\psi}^{1},\bar{\psi}^{2}) $ satisfies $ J_{1}(x_{0},\bar{\psi}^{1},\bar{\psi}^{2}) < J_{1}(x_{0},\hat{\psi}^{1},\hat{\psi}^{2}) $. Clearly, the leader can resolve  multiplicities in  \eqref{eq:leadersproblem} by enforcing admissible solutions at every time instant.  
\end{remark}
 
 \subsection{ Sufficient conditions for FSN solution}\label{sec:FSNValue} 
  In this subsection, we establish sufficient conditions for the existence of an FSN solution.  
  Building on  the existence and  uniqueness of the Nash equilibrium in simultaneous interactions (see Lemma \ref{lem:nLCP}), the next theorem  utilizes the  reachable sets $\mathsf{X}_{k}$, $k \in \mathsf{K}$  to transform 
  	conditions \eqref{eq:Nash1K}-\eqref{eq:leaderannouncement} in Definition \ref{def:CFS} into a recursive characterization of FSN solution.   

\begin{theorem}\label{th:VerificationTh1}
 Consider the CNZDG described by \eqref{eq:state}-\eqref{eq:objective}. Let Assumptions \ref{ass:A1}, \ref{ass:seperability} and \ref{ass:convex} hold true. If there exist functions $W^{i}(k,\,.\,):\mathsf{X}_{k}\rightarrow \mathbb{R}$, $k\in \mathsf{K}~i\in \{1,2\}$ such that for all $k\in \mathsf{K}_{l}$ and $\gamma _{k}^{1}\in \Gamma _{k}^{1},~g_{u,k}^{2}(x_{k},\gamma_{k}^{1}(x_{k}),\cdotp)+W^{2}(k+1,f_{k}(x_{k},\gamma _{k}^{1}(x_{k}),\cdotp)) $ is strictly convex function on $\mathsf{U}_{k}^{2}$ and for all $k\in \mathsf{K},~W^{i}(k,\cdotp)$ satisfy the following (backward) recursive relations:
 	\begin{enumerate}
 		\item At period $K$
 		\begin{subequations}\label{eq:valfuncon}
 		\begin{align} W^i(K,x_K)=g_{K}^i(x_K,(v_K^{1\star},v_K^{2\star})),~i\in\{1,2\},
 			\label{eq:termcongame}
 		\end{align} 
           where $(\mathsf{v}_{K}^{\star },{\upmu}_{K}^{\star })=\texttt{SOL}(\mathrm{pCP}(x_{K}))$.

           \item At periods $k=K-1,\cdots ,1,0$ 
 		\begin{align}
 			&W^i(k, x_k) =  g_k^i\big(x_k, \gamma_k^{1*}(x_k),   \gamma_k^{2*}  (x_k), (v_k^{1\star}, v_k^{2\star})\big)\notag\\
 			&~+W^i\bigl(k+1, f_k(x_k, \gamma_k^{1*}(x_k),  \gamma_k^{2*}  (x_k))\bigr),~i\in\{1,2\},   \label{eq:leadersproblemrecursive}
 		\end{align}
 	where
 	\begin{align} 
 		&(\mathsf v_k^\star,{\upmu}_k^{\star})=\texttt{SOL}(\mathrm{pCP}_k(x_k)), \label{eq:Nash3rdstage}\\
 					&\mathsf{R}^2_{k}(\gamma_{k}^1)= \underset{\gamma_{k}^{2} \in \Gamma_k^2}{\mathrm{arg min}}\,\Big\{g_{u, k}^2\bigl(x_k, \gamma_k^1(x_k), \gamma_k^2(x_k)\bigr)\notag\\
                  &\hspace{0.5in} +W^2\bigl(k+1, f_k(x_k, \gamma_k^1(x_k), \gamma_k^2(x_k))\bigr) \Big\} ,  \label{eq:RRinsubgameproblem} \\
 					&\gamma_k^{1\star} \in \underset{\gamma_{k}^{1} \in \Gamma_k^1}{\mathrm{arg min}}\,\Big\{g_{u, k}^1\bigl(x_k, \gamma_k^1(x_k), (\mathsf R_k^2 \circ \gamma_k^1)(x_k)\bigr)\notag \\
 				&\hspace{0.3in}+W^1\bigl(k+1, f_k(x_k, \gamma_k^1(x_k), (\mathsf R_k^2 \circ \gamma_k^1)(x_k))\bigr) \Big\}, \label{eq:leadersproblemfirst}\\
 					&\gamma_k^{2\star}=\mathsf R_k^2(\gamma_k^{1\star}).\label{eq:followerresponse1}
 		\end{align} 
 	\end{subequations}
 	\end{enumerate}
        Then, the pair of strategies $\{\psi ^{i\star }=((\gamma _{k}^{i\star}(x_{k}),v_{k}^{i\star })_{K_{l}},v_{K}^{i\star }),i=1,2\}$ constitutes an FSN solution for CNZDG. Further, every such FSN solution is strongly time consistent.
\end{theorem}

\begin{proof}
       We prove the theorem by backward induction. At terminal time \( K \), the inequalities \eqref{eq:Nash1K}-\eqref{eq:Nash2K} must hold for all admissible actions \( \{\psi_{\lbrack 0,K-1]}^{i},~i\in \{1,2\}\} \). Consequently, they must also hold for all reachable states \( x_K \) resulting from these actions. Thus, conditions \eqref{eq:Nash1K}-\eqref{eq:Nash2K} reduce to a static game where each player \( i \in \{1,2\} \) solves  
       \begin{align}
       	v_K^{i\star}\in \argmin_{v_K^{i} \in \V_K^i(v_k^{j\star})}~ g_{K}^i(x_K, (v_K^{i}, v_K^{j\star})), \label{eq:NashatK}
       \end{align}  
       where $j\in\{1,2\}$, $j\neq i$.
       By Assumption \ref{ass:convex} and Lemma \ref{lem:nLCP}, the solution of \eqref{eq:NashatK}, given by \( (\bv_{K}^{\star },{\upmu}_{K}^{\star })=\texttt{SOL}(\mathrm{pCP}_{K}(x_{K})) \), is unique and provides the FSN solution of the subgame at \( K \). Moreover, the FSN decisions \( (v_{K}^{1\star },v_{K}^{2\star }) \) depend only on the current state \( x_K \), not on past state values (including \( x_0 \)).  
       
        Next, we assume the theorem holds up to time period \( k+1 \in \mathsf{K}_r \) backward in time. That is, the strategies \( \{\psi _{\lbrack k+1,K]}^{i\star},~i\in \{1,2\}\} \) obtained from \eqref{eq:leadersproblemrecursive}-\eqref{eq:followerresponse1} satisfy Definition \ref{def:CFS} of an FSN solution. The FSN cost accumulated by player \( i \in \{1,2\} \) for any \( x_{k+1} \in \mathsf{X}_{k+1} \) is  $$
       W^i(k+1, x_{k+1}) =  \sum_{\tau=k+1}^{K}g_{\tau}^i(x_{\tau}, \psi_\tau^{1\star}, \psi_\tau^{2\star}).$$
       We now show that the theorem holds at time \( k \). The inequalities \eqref{eq:nash1}--\eqref{eq:leaderannouncement}, defined at time \( k \) while fixing downstream decisions at the FSN solutions \( \{\psi _{[k+1,K]}^{i\star },~i\in \{1,2\}\} \), correspond to a quasi-hierarchical interaction with three decision stages. These inequalities must hold for all admissible actions \( \{\psi _{\lbrack 0,k-1]}^{i},~i\in \{1,2\}\} \), and consequently, for all states \( x_k \in \mathsf{X}_k \) that are reachable through a combination of these actions.         
       
       Between time instants \( k \) and \( k+1 \), players interact in three decision stages, where each player \( i \in \{1,2\} \) minimizes the cost functional  $$
       g_k^i(x_k, (u_k^{1}, u_k^2), (v_k^1, v_k^{2}))+W^i(k+1, x_{k+1}),$$
       subject to constraints \( h_{k}^{i}(x_{k},\bv_{k})\geq 0,\,v_{k}^{i}\geq 0 \), with \( x_{k+1}=f_{k}(x_{k},u_{k}^{1},u_{k}^{2}) \). From separability Assumption \ref{ass:seperability} and Lemma \ref{lem:nLCP}, the unique Nash outcome \( \bv_k^\star \) of the third-stage problem is given by \eqref{eq:Nash3rdstage}.         
       In the second stage, under the strict convexity assumption of \( g_{u,k}^{2}(x_{k},\gamma _{k}^{1}(x_{k}),\cdotp)+W^{2}(k+1,f_{k}(x_{k},\gamma _{k}^{1}(x_{k}),.) \) on \( \mathsf{U}_{k}^{2} \), the optimal response \( \mathsf{R}_{k}^{2} \), defined by \eqref{eq:RRinsubgameproblem}, is unique. In the first stage, using the follower's response, the leader's problem is given by \eqref{eq:leadersproblemfirst}. Thus, \( W^i(k,x_k) \), given by \eqref{eq:leadersproblemrecursive}, precisely represents the FSN cost for player \( i \) in the subgame starting at \( (k,x_k) \).  
       We note that FSN decisions at each time instant \( k \) depend only on the current \( x_k \) and not on past state values (including \( x_0 \)). The strong time consistency of the FSN solution follows directly from its backward recursive construction.         
\end{proof} 
\begin{remark}
       We note that Theorem \ref{th:VerificationTh1} provides only sufficient conditions for the existence of an FSN solution for CNZDG. This follows from the strict convexity of \( g_{u,k}^{2}(x_{k},\gamma _{k}^{1}(x_{k}),\cdotp) + W^{2}(k+1,f_{k}(x_{k},\gamma_{k}^{1}(x_{k}),\cdotp)) \) over \( \mathsf{U}_{k}^{2} \) and the DSC Assumption \ref{ass:convex} on cost functions \( g_{v,k}^{i} \), which ensure a unique follower response and a unique Nash equilibrium for the static game \( \bm{\Uplambda}_{k}(x_{k}) \) at each time instant.
\end{remark}
 \subsection{Parametric feedback Stackelberg solution and its relation to  FSN solution}\label{sec:FSNfrompFS}
 
       In this subsection, we show that an FSN solution is closely related to a feedback Stackelberg solution of an unconstrained parametric nonzero-sum difference game (pNZDG) involving only sequential interactions. To this end, we first define a pNZDG associated with the CNZDG, with parameters $\{\mathsf{w}_{\tau}:=(w_{\tau }^{1},w_{\tau }^{2})\in \mathbb{R}^{s_{1}+s_{2}},{\uptheta}_{\tau}:=(\theta_{\tau}^{1},\theta _{\tau}^{2})\in \mathbb{R}^{c_{1}+c_{2}},\tau \in \mathsf{K}\}$ as follows: 
 \begin{subequations}\label{eq:pnzdg} 
 	\begin{align}
 		& \mathrm{pNZDG}: ~\min_{\tilde{\bu}^i} \Bigl\{\bar{J}_i\big(x_0,\tilde{\bu}^1,\tilde{\bu}^2;\{(\bw_{\tau}, {\uptheta}_{\tau}) \}_{\tau \in \K}\big)\notag\\
         &\quad= g_{K}^i(x_K, (w_K^1, w_K^2))-\theta_K^{i'}h_K^i(x_K, \bw_K)\notag\\
 		& +\sum_{k=0}^{K-1}\big(g_{k}^i((x_k,(u_k^1, u_k^2)), (w_k^1, w_k^2) )-\theta_k^{i'}h_k^i(x_k, \bw_k)\big)\Bigr\},\label{eq:Pobjective}\\ 
 		&  \text{subject to} ~x_{k+1} = f_k(x_k, u_k^1, u_k^{2}),~ k \in \K_l,~ x_0~\text{given}.\label{eq:Pstate}
 	\end{align}
 \end{subequations}
        In pNZDG, the players interact sequentially in the decision variables $(u_{k}^{1},u_{k}^{2})$ at every time instant $k\in \mathsf{K}_{l}$. Under feedback information structure, the control action $u_{k}^{i}$ of player $i$ at instant $k$ is denoted by $u_{k}^{i}:=\xi _{k}^{i}(x_{k})\in \mathsf{U}_{k}^{i}$, where $\xi _{k}^{i}:\mathbb{R}^{n}\rightarrow \mathsf{U}_{k}^{i}$ is a measurable mapping with the set of all such mappings be denoted by $\Xi_{k}^{i}$. The feedback strategy profile of player $i$ be denoted by $\xi^{i}$ and the corresponding strategy set by $\Xi ^{i}=\prod_{k=0}^{K-1}\Xi_{k}^{i}$. As there are no constraints, the reachable set $\bar{\mathsf{X}}_{k}$ is entire $\mathbb{R}^{n}$. The definition of parametric feedback Stackelberg solution for pNZDG follows from the standard feedback Stackelberg solution \cite{Basar:99}, which is given as follows.
 \begin{definition}[{\cite[Definition 3.29]{Basar:99}}]
 	\label{def:pFS}
        A pair $(\xi^{1\star}, \xi^{2\star}) \in (\Xi^1, \Xi^2)$ constitutes a \emph{parametric feedback Stackelberg} (pFS) solution for pNZDG if the following conditions are satisfied
\begin{enumerate}
       \item  For all $\xi^i_{[0,k-1]}\in \Xi_{[0,k-1]}^i$,~$i\in\{1,2\}$ with $k=K-1,\cdots,1$
 	\begin{subequations}\label{eq:pFS}
 		\begin{align}
 			&\bar{J}_1\bigl(x_0, (\xi^1_{[0,k-1]}, \xi_k^{1\star}, \xi^{1\star}_{[k+1,K-1]}), \notag\\
        &\qquad(\xi^2_{[0,k-1]}, {\xi}_k^{2\star}, \xi^{2\star}_{[k+1,K-1]});\{(\bw_{\tau}, {\uptheta}_{\tau}) \}_{\tau \in \K}\bigr)\notag\\
 			&= 
 			\optim{ {\xi}^1_k\in \Xi^1_k}{min}~
 			\max_{\hat{\xi}_k^2\in\bar{\mathsf{R}}^2_k(\xi^1_k) }\bar{J}_1\bigl(x_0, (\xi^1_{[0,k-1]},\xi^{1}_k, \xi^{1\star}_{[k+1,K-1]}),\notag\\
    &\qquad(\xi^2_{[0,k-1]},\hat{\xi}_k^2, \xi^{2\star}_{[k+1,K-1]});\{(\bw_{\tau}, {\uptheta}_{\tau}) \}_{\tau \in \K}\bigr),\label{eq:recursiveleaderp}
 		\end{align}
 where $\bar{\mathsf{R}}^2_k(\xi^1_k)$ is the optimal response set of the follower at time $k$, defined by
 		\begin{align}
 			\bar{\mathsf{R}}^2_k(\xi^1_k) :&=  
 			 \argmin_{\xi_k^2\in \Xi_k^2}\bar{J}_2\bigl(x_0, (\xi^1_{[0,k-1]},\xi_k^1,\xi^{1\star}_{[k+1,K-1]}),\notag\\
 			&(\xi^2_{[0,k-1]}, {\xi}_k^2, \xi^{2\star}_{[k+1,K-1]});\{(\bw_{\tau}, {\uptheta}_{\tau}) \}_{\tau \in \K}\big),
 			\label{eq:rationalreactionsetp}
 		\end{align} 
 	\label{eq:recursivepfs}
 	\end{subequations}
 	\item The optimal response set $\bar{\mathsf R}_k^2(\xi_k^{1\star})$ is a 
          singleton set. \label{itm:requirementp}
 	\end{enumerate}
       Further, $J_1(x_0, \xi^{1\star}, \xi^{2\star};\{(\mathsf{w}_{\tau}{\uptheta}_{\tau}) \}_{\tau \in \mathsf{K}})$ and $J_2(x_0, \xi^{1\star}, \xi^{2\star};$ $\{(\mathsf{w}_{\tau}, {\uptheta}_{\tau}) \}_{\tau \in \mathsf{K}})$ represent the parametric feedback Stackelberg costs incurred by the leader and the follower, respectively, with parameters $\{(\mathsf{w}_{\tau}, {\uptheta}_{\tau}) \}_{\tau \in \mathsf{K}}$.
\end{definition}

 From Assumption \ref{ass:seperability}, the cost functions \( g_k \), \( k \in \mathsf{K}_l \), in \eqref{eq:Pobjective} have a separable structure, that is, they do not contain cross-terms involving the decision variables $(u_k^1,u_k^2)$ and the parameters $(w_k^1,w_k^2)$.  The recursive formulation of a pFS solution for pNZDG follows from \eqref{eq:recursiveleaderp} and \eqref{eq:rationalreactionsetp}.
        Using the standard feedback-Stackelberg solution \cite[Theorem 7.2]{Basar:99}, the next lemma characterizes the structure of a pFS solution for pNZDG. We omit its proof, as it follows directly from \cite[Theorem 7.2]{Basar:99}.

 \begin{lemma}\label{lem:VerificationTh2}
       Consider the pNZDG described by \eqref{eq:pnzdg}. Let Assumption \ref{ass:seperability} hold. If there exist functions $W_{p}^{i}(k,\,.\,):\bar{\mathsf{X}}_{k}\rightarrow \mathbb{R}$, $k\in \mathsf{K},~i\in \{1,2\}$ such that for all $k\in \mathsf{K}_{l}$ and $\forall \xi _{k}^{1}\in \Xi_{k}^{1},~g_{u,k}^{2}(x_{k},\xi _{k}^{1}(x_{k}),\cdotp)+W_{p}^{2}(k+1,f_{k}(x_{k},\xi _{k}^{1}(x_{k}),\cdotp))$ is strictly convex on $\mathsf{U}_{k}^{2}$ $(k\in \mathsf{K}_{l})$ and $\forall k\in \mathsf{K},~W_{p}^{i}(k,\cdotp)$ satisfy the following (backward) recursive relations:
 		\begin{enumerate}
 		\item At period $K$
 		\begin{subequations}\label{eq:valfunpar}
 		\begin{align}	
       & W_p^i(K, x_K;(\bw_{K}, {\uptheta}_{K}))\notag\\
       &=g_{K}^i(x_K, \bw_K)-{\theta_K^{i}}^\prime h_K^i(x_{K}, \bw_{K}),~ i \in \{1,2\}. \label{eq:PWatK}  \end{align}  
 		\item At periods $k=K-1,\cdots,1,0$
 		\begin{align}
 	&W_p^i\big(k, x_k;\{(\bw_{\tau}, {\uptheta}_{\tau})\}_{\tau=k}^K\big)= g_{k}^i(x_k, \xi_k^{1\star}(x_k),  \xi_k^{2\star}(x_k), \bw_k)\notag\\
     &\quad-{\theta_k^{i}}^\prime h_k^i(x_k, \bw_k)+W_p^i\bigl(k+1, f_k(x_k, \xi_k^{1\star}(x_k), \xi_k^{2\star}(x_k));\notag\\
 	&\hspace{1.75in} \{(\bw_{\tau}, {\uptheta}_{\tau})\}_{\tau=k+1}^K\bigr) ,\label{eq:WPleaderrecursive1}
 		\end{align}
 		where
 		\begin{align} 
 			&\bar{\mathsf{R}}_k^2
 		(\xi_k^1)= \optim{\xi_k^{2} \in \Xi^2_k}{arg min}\Big\{g_{u,k}^2(x_k,  \xi_k^{1}(x_k), \xi_k^{2}(x_k) ) +W_p^2\bigl(k+1, \nonumber\\
 		&\hspace{0.45in} f_k(x_k, \xi_k^{1}(x_k), \xi_k^{2}(x_k));\{(\bw_{\tau}, {\uptheta}_{\tau})\}_{\tau=k+1}^K\bigr) \Big\},\label{eq:ParametricRR}\\
 			& \xi_k^{1\star}\in \optim{\xi_k^{1} \in \Xi^1_k}{argmin}~\Big\{g_{u,k}^1\bigl(x_k, \xi_k^{1}(x_k), (\bar{\mathsf{R}}_k^2\circ \xi_k^1)(x_k) \bigr)\notag\\
 		&\hspace{0.8 in} +W_p^1\bigl(k+1, f_k\bigl(x_k, \xi_k^{1}(x_k), (\bar{\mathsf{R}}_k^2\circ \xi_k^1)(x_k)\bigr);\notag\\
   &\hspace{1.25in}\{(\bw_{\tau}, {\uptheta}_{\tau})\}_{\tau=k+1}^K\bigr) \Big\},\label{eq:WPleaderrecursive}\\
 		&\xi_k^{2\star}=\bar{\mathsf{R}}_k^{2}
 		(\xi_k^{1\star}).\label{eq:pfollower}
 		\end{align}  
 \end{subequations}
 	\end{enumerate}
       Then, the strategy pair $\{\xi^{1\star},~\xi^{2\star}\}$ constitutes a pFS solution for pNZDG. Further, every such solution is strongly time consistent, and has a parametric representation given by 
 	\begin{align}
 		\big\{\xi_k^{i\star}(x_k;\{(\bw_{\tau}, {\uptheta}_{\tau})\}_{\tau=k+1}^K),\, k \in \K_l,~ i \in \{1,2\} 
 		\big\}.\label{eq:Parametricsolutions}
 	\end{align}  
 \end{lemma} 
 \begin{remark}\label{rem:PValuefnc}
       We notice that for any time $k\in \mathsf{K}_{l}$, and any $x_{k}\in \bar{\mathsf{X}}_{k}$, $W_{p}^{i}(k,x_{k};\{(\mathsf{w}_{\tau },{\uptheta}_{\tau})\}_{\tau =k}^{K})$ denotes the cost incurred by player $i$ using the pFS strategies $(\xi _{\lbrack k,K-1]}^{1\star },\xi _{[k,K-1]}^{2\star})$ and thus it represents the value function for player $i$. Here, the value functions in \eqref{eq:WPleaderrecursive1} are obtained by solving the game recursively backwards, and depend on the downstream parameters. For this reason, we represent the value function at time instant $k$ explicitly in parametric form as $W_{p}^{i}(k,x_{k};\{(\mathsf{w}_{\tau },{\uptheta}_{\tau})\}_{\tau =k}^{K})$.  Consequently, the pFS solution in \eqref{eq:Parametricsolutions}, derived from \eqref{eq:ParametricRR}-\eqref{eq:WPleaderrecursive}, also depends on the downstream parameters.
\end{remark}
       Next, we present the main result of this section where we show that an FSN solution for CNZDG can be obtained from a pFS solution for pNZDG through a specific choice of the parameters.
 \begin{theorem}\label{thm:equivtheorem}
       Consider the CNZDG described by \eqref{eq:state}-\eqref{eq:objective} and the pNZDG described by \eqref{eq:pnzdg}. Let Assumptions \ref{ass:A1}, \ref{ass:seperability} and \ref{ass:convex} hold true. Assume there exist functions $W^{i}(k,\,.\,):\mathsf{X}_{k}\rightarrow \mathbb{R}$ and $W_{p}^{i}(k,\,.\,):\bar{\mathsf{X}}_{k}\rightarrow \mathbb{R}$ for $k\in \mathsf{K}$, $i\in \{1,2\}$ satisfying conditions \eqref{eq:valfuncon} and \eqref{eq:valfunpar}, respectively. Further, assume that for all $k\in  \mathsf{K}_{l}$ the functions $g_{u,k}^{2}(x_{k},\gamma _{k}^{1}(x_{k}),\cdotp)+W^{2}(k+1,f_{k}(x_{k},\gamma _{k}^{1}(x_{k}),\cdotp)),~\forall \gamma _{k}^{1}\in \Gamma _{k}^{1}$ and $g_{u,k}^{2}(x_{k},(\xi_{k}^{1}(x_{k}),\cdotp),\cdotp)+W_{p}^{2}(k+1,f_{k}(x_{k},\xi_{k}^{1}(x_{k}),\cdotp)),~\forall \xi _{k}^{1}\in \Xi _{k}^{1}$ are strictly convex on $\mathsf{U}_{k}^{2}$. Let for player $i\in \{1,2\}$, $\{\xi_{k}^{i\star }(x_{k};\{(\mathsf{w}_{\tau },{\uptheta}_{\tau })\}_{\tau=k}^{K}),\,k\in \mathsf{K}_{l}$ denote a pFS solution for PNZDG. Then, an FSN solution for player $i\in \{1,2\}$ satisfies  
 \begin{subequations}\label{eq:Thequ}
 	\begin{align}
 		&\gamma_k^{i\star}(x_k) = \xi_k^{i\star}(x_k;\{(\bv_{\tau}^{\star}, {\upmu}_{\tau}^{\star})\}_{\tau=k+1}^K),~\forall k \in \K_l,\label{eq:equivofsol}
 	\end{align} 
       where $(\bv^{\star}_{k}, {\mu}_{k}^{\star}) = \texttt{SOL}(\mathrm{pCP}_{k}(x_{k}))$ for all $k\in \mathsf K$, and $\{x_k,~ k \in \K\}$ is the state trajectory generated by the difference equation
 	\begin{align*}
 		x_{k+1} = f_k\big(x_k, \xi_k^{1\star}(x_k;\{(\bv_{\tau}^{\star}, {\upmu}_{\tau}^{\star})\}_{\tau=k+1}^K), \xi_k^{2\star}(x_k;\{(\bv_{\tau}^{\star}, {\upmu}_{\tau}^{\star})\}_{\tau=k+1}^K)\big).  
 	\end{align*}
       Further, for all $k\in \mathsf K$  and $i\in \{1,2\}$ the (value) function $W^i(k,.): \mathsf{X}_k\rightarrow \mathbb R$   satisfies 
 	\begin{align}
 		W^i(k, x_k) = W_p^i(k, x_k;\{(\bv_{\tau}^{\star}, {\upmu}_{\tau}^{\star})\}_{\tau=k}^K),\quad x_k \in \mathsf{X}_k.\label{eq:WequivWp}
 	\end{align}
 \end{subequations}
 \end{theorem}
 \begin{proof}
 	 We prove the theorem using the backward induction principle. The terminal game associated with CNZDG at time instant $K$, with $x_{K} \in \mathsf{X}_{K}$, is characterized by \eqref{eq:termcongame}. By Assumption \ref{ass:convex} and Lemma \ref{lem:nLCP}, there exists a unique $(\bv_{K}^{\star}, {\upmu}_{K}^{i\star }) = \texttt{SOL}(\mathrm{pCP}_{K}(x_{K}))$. Setting $(\bw_{K}, {\uptheta}_{K}) = (\bv_{K}^{\star}, {\upmu}_{K}^{\star })$ in \eqref{eq:PWatK} and using the FSN costs \eqref{eq:termcongame} along with the complementarity condition ${\mu_{K}^{i\star}}^{\prime }h_{K}^{i}(x_{K}, \bv_{K}^{\star}) \equiv 0$ for $i \in \{1,2\}$, we obtain
 	\begin{align*}
 		W^i(K, x_K) &= g_{K}^i(x_K, (v_K^{1\star}, v_K^{2\star})) - {\mu_K^{i\star}}^\prime h_K^i(x_{K}, \bv_{K}^{\star})
 		\\&= W_p^i(K, x_K; (\bv_{K}^{\star}, {\upmu}_{K}^{\star})).
 	\end{align*} 
 	Thus, the theorem holds at terminal time.   Next, assuming that the theorem holds for time $k+1$, we show it holds for $k$, where $k = K-1, \dots, 0$. By the induction hypothesis, for $l = k+1$ to $l = K-1$, the FSN solution of player $i \in \{1,2\}$ relates to the pFS solution as
 	\begin{align}
 		\gamma_{l}^{i\star}(x_l) = \xi_l^{i\star}(x_l; \{(\bv_{\tau}^{\star}, {\upmu}_{\tau}^{\star})\}_{\tau=l+1}^K).
 	\end{align}
 	The FSN cost of player $i$ in the subgame starting at $(k+1, x_{k+1})$ is given by
 	\begin{align}
 		W^i(k+1, x_{k+1}) = W_p^i(k+1, x_{k+1}; \{(\bv_{\tau}^{\star}, {\upmu}_{\tau}^{\star})\}_{\tau=k+1}^K), \label{eq:PiWKn}
 	\end{align}
 	where $(\bv_{\tau }^{\star }, {\upmu}_{\tau }^{\star }) = \texttt{SOL}(\mathrm{pCP}_{\tau }(x_{\tau }))$ for $\tau = k+1, \dots, K$.
 	 Now, consider the static game at time $k$ characterized by \eqref{eq:nash1}-\eqref{eq:leaderannouncement} for any $x_k \in \mathsf{X}_k$, with player $i$'s cost function
\begin{align}
	g_k^i(x_k, (u_k^{1}, u_k^2), (v_k^1, v_k^{2})) + W^i(k+1, f_k(x_k, u_k^1, u_k^2)).
	\label{eq:playericostinstantk}
\end{align}
       From Assumption \ref{ass:convex}, and Lemma \ref{lem:nLCP}, the third stage decisions are obtained as the unique Nash outcome given by $(\bv_{k}^{\star},{\upmu}_{k}^{\star })=\texttt{SOL}(\mathrm{pCP}_{k}(x_{k}))$. Using \eqref{eq:PiWKn} in the second and first stage decision problems \eqref{eq:RRinsubgameproblem} and \eqref{eq:leadersproblemfirst}, respectively, we get 
 \begin{align} 
       &\mathsf{R}^2_{k}(\gamma_{k}^1)= \underset{\gamma_{k}^{2} \in \Gamma_k^2}{\mathrm{arg min}}\,\Big\{g_{u, k}^2\bigl(x_k, \gamma_k^1(x_k), \gamma_k^2(x_k)\bigr)\notag \\
       &~+W^2_p\bigl(k+1, f_k(x_k, \gamma_k^1(x_k), \gamma_k^2(x_k));\{(\bv_{\tau}^{\star}, {\upmu}_{\tau}^{\star})\}_{\tau=k+1}^K\bigr) \Big\},
 \label{eq:P2RRCKn}\\
       &\gamma_k^{1\star}\in \underset{\gamma_{k}^{1} \in \Gamma_k^1}{\mathrm{arg min}}\,\Big\{g_{u, k}^1\bigl(x_k, \gamma_k^1(x_k), (\mathsf R_k^2 \circ \gamma_k^1)(x_k)\bigr)+W^1_p\bigl(k+1,\notag \\
       &\hspace{0.45in} f_k(x_k, \gamma_k^1(x_k), (\mathsf R_k^2 \circ \gamma_k^1)(x_k));\{(\bv_{\tau}^{\star}, {\upmu}_{\tau}^{\star})\}_{\tau=k+1}^K\bigr) \Big\}. \label{eq:leadercon1}
 \end{align} 
 
        Now, consider the subgame starting at $(k,x_{k})$ for pNZDG with the parameters fixed as $(\bw_{\tau },{\uptheta}_{\tau })=(\bv_{\tau }^{\star },{\upmu}_{\tau }^{\star })=\texttt{SOL}(\mathrm{pCP}_{\tau }(x_{\tau }))$ for $\tau =k+1,\cdots ,K$. The follower's rational response \eqref{eq:ParametricRR}, and the leader's optimization problem \eqref{eq:WPleaderrecursive} are given by 
\begin{align}
 			&\bar{\mathsf{R}}_k^2
 			(\xi_k^1)=\underset{\xi_{k}^{2} \in \Xi_k^2}{\mathrm{arg min}}\,\Big\{g_{u, k}^2\big(x_k, \xi_k^{1}(x_k), \xi_n^{2}(x_k)\big) \nonumber\\
 			&~+W_p^2\bigl(k+1,f_k(x_k, \xi_n^{1}(x_k), \xi_n^{2}(x_k));\{(\bv_{\tau}^{\star}, {\upmu}_{\tau}^{\star})\}_{\tau=k+1}^K\bigr) \Big\},\label{eq:P2RRWpCkn} \\
 			 &\xi_k^{1\star}\in \optim{\xi_k^{1} \in \Xi^1_k}{argmin}~\Big\{g_{u,k}^1\bigl(x_k, \xi_k^{1}(x_k), (\bar{\mathsf{R}}_k^2\circ \xi_k^1)(x_k)\bigr)+W_p^1\bigl(k+1,\notag\\
 			&\hspace{0.35in} f_k\bigl(x_k, \xi_k^{1}(x_k), (\bar{\mathsf{R}}_k^2\circ \xi_k^1)(x_k)\bigr); \{(\bv_{\tau}^{\star}, {\upmu}_{\tau}^{\star})\}_{\tau=k+1}^K\bigr)\Big\}.
 			 \label{eq:leaderpar1}
\end{align}
              
           In the follower's optimization problems \eqref{eq:P2RRCKn} and \eqref{eq:P2RRWpCkn}, the strict convexity of the objective function over $\mathsf{U}_{k}^{2}$ ensures that the optimal reaction sets $\mathsf{R}_{k}^{2}(\gamma _{k}^{1})$ and $\bar{\mathsf{R}}_{k}^{2}(\xi _{k}^{1})$ are singletons. Since these functions share the same form, it follows that ${\mathsf{R}}_{k}^{2}(\gamma_{k}^{1}) = \bar{\mathsf{R}}_{k}^{2}(\gamma _{k}^{1})$ for all $\gamma_{k}^{1} \in \Gamma _{k}^{1}$. Thus, the follower's response remains unchanged across both problems for any leader's announcement in $\Gamma _{k}^{1}$.  Similarly, in the leader's optimization problems \eqref{eq:leadercon1} and \eqref{eq:leaderpar1}, the objective functions are identical. Restricting the leader's strategy space in these problems to $\Gamma_{k}^{1} \subset \Xi _{k}^{1}$ therefore yields the same optimal strategy sets.

         For pNZDG, Lemma \ref{lem:VerificationTh2} establishes that the pFS solution of \eqref{eq:P2RRWpCkn} and \eqref{eq:leaderpar1} takes the form $\xi_k^{i\star}(x_k; \{(\bv_{\tau}^{\star}, {\upmu}_{\tau}^{\star})\}_{\tau=k+1}^K) \in \Xi_k^i$ for $i=1,2$. The downstream parameters in these strategies are set as $(\bw_{\tau}, {\uptheta}_{\tau}) = (\bv^{\star}_{\tau}, {\upmu}_{\tau}^{\star}) = \texttt{SOL}(\mathrm{pCP}_{\tau}(x_{\tau}))$ for $\tau = k+1, \dots, K$. At time $k$, using the pFS solutions, we obtain $x_{k+1} = f(x_k, \xi_k^{1\star}(x_k), \xi_k^{2\star}(x_k))$, which must satisfy $\texttt{SOL}(\mathrm{pCP}_{\tau}(x_{k+1})) \neq \emptyset$ since the downstream parameters are fixed. From Assumption \ref{ass:A1}.\ref{itm:1}, this ensures that the pFS strategies at time $k$ remain admissible, i.e., $\xi_k^{i\star}(x_k; \{(\bv_{\tau}^{\star}, {\upmu}_{\tau}^{\star})\}_{\tau=k+1}^K) \in \Gamma_k^i$ for $i=1,2$. Finally, given the uniqueness of the follower's response in \eqref{eq:P2RRCKn} and \eqref{eq:P2RRWpCkn}, and the coincidence of the leader's optimal announcement sets over $\Gamma_k^1$ in \eqref{eq:leadercon1} and \eqref{eq:leaderpar1}, the FSN solution at time $k$ follows from the pFS solution as 
         
 \begin{align}
	\gamma_k^{i\star}(x_k)=\xi_k^{i\star}(x_k;\{(\bv_{\tau}^{\star}, {\upmu}_{\tau}^{\star})\}_{\tau=k+1}^K),~i=1,2.\label{eq:P1equality}
\end{align}
        The value function of player $i$ in the subgame starting at $(k,x_k)$ for CNZDG is given by \eqref{eq:leadersproblemrecursive}. From \eqref{eq:PiWKn}, incorporating the complementarity condition $\mu_{k}^{i\star^{\prime}}h_k^i(x_k, \bv_k^{\star}) = 0$, we obtain
 \begin{align*}
       W^i(k, x_k)& =  g_k^i\big(x_k, \gamma_k^{1*}(x_k),   \gamma_k^{2*}  (x_k), (v_k^{1\star}, v_k^{2\star})\big)-\mu_{k}^{i\star'}h_k^i(x_k, \bv_k^{\star})\notag\\
  	&\quad +	W_p^i\big(k+1, x_{k+1};\{(\bv_{\tau}^{\star}, {\upmu}_{\tau}^{\star})\}_{\tau=k+1}^K\big),\notag \\
  	&=W_p^i\big(k, x_k;\{(\bv_{\tau}^{\star}, {\upmu}_{\tau}^{\star})\}_{\tau=k}^K\big).
\end{align*}
       The last equality follows from \eqref{eq:WPleaderrecursive1} by setting $(\bw_k, {\uptheta}_k)=(\bv^{\star}_k, {\upmu}_k^{\star})= \texttt{SOL}(\mathrm{pCP}_{k}(x_{k}))$. Thus, the theorem follows by backward induction.  
 \end{proof}
 Theorem \ref{thm:equivtheorem} is constructive and establishes a relationship between the pFS solution for pNZDG and the FSN solution for CNZDG. Using this, we characterize the FSN solution set. Substituting pFS strategies into \eqref{eq:Pstate} yields the pFS state trajectory:  
 \begin{align}
 	x_{k+1} &= f_k\big(x_k, \xi_k^{1\star}(x_k;\{(\bw_{\tau}, {\uptheta}_{\tau})\}_{\tau=k+1}^K), \notag \\
 	&\hspace{0.95in}\xi_k^{2\star}(x_k;\{(\bw_{\tau}, {\uptheta}_{\tau})\}_{\tau=k+1}^K)\big),~ k \in \K_l. \label{eq:pStatetrajectory}
 \end{align}
 For a given initial condition $x_{0} \in \mathbb{R}^{n}$, this trajectory is fully determined by the choice of parameters \tb{$(\mathsf{w}_{\tau },{\uptheta}_{\tau })\in \mathbb{R}_{+}^{s+c}$, $c:=c_1+c_2$, $\tau \in \mathsf{K}$}. Defining $ x_{\mathsf{K}} := \mathrm{col}(x_{k})_{k=1}^{K}$, $\mathsf{w}_{\mathsf{K}} := \mathrm{col}(\mathsf{w}_{k})_{k=1}^{K}$, and ${\uptheta}_{\mathsf{K}} := \mathrm{col}({\uptheta}_{k})_{k=1}^{K}$, the pFS state trajectory satisfies  
 \begin{align} 
 	x_\mathsf K= \mathcal X_\text{pFS}((\mathsf w_\mathsf K, \boldsymbol \uptheta_\mathsf K);x_0),\label{eq:Xvectoreq}
 \end{align}
 where $\mathcal{X}_{\text{pFS}}:\mathbb{R}_{+}^{(s+c)K}\rightarrow \mathbb{R}^{nK}$ is the vector representation of \eqref{eq:pStatetrajectory} obtained by eliminating $x_{k}$ on the right-hand side. Next, using \eqref{eq:nLCP}, we define $
 	\nabla {L}_{\mathsf{K}}(x_{\mathsf{K}},\mathsf{w}_{\mathsf{K}},{\uptheta}_{\mathsf{K}}) := \mathrm{col}(\nabla {L}_{k}(x_{k},\mathsf{w}_{k},{\uptheta}_{k}))_{k=1}^{K}$ and
 	$h_{\mathsf{K}}(x_{\mathsf{K}},\mathsf{w}_{\mathsf{K}}) := \mathrm{col}(h_{k}(x_{k},\mathsf{w}_{k}))_{k=1}^{K}$. We then formulate the following complementarity problem:
 \begin{align} 
 	&\mathrm{pCP}(x_\K): 0 \leq\begin{bmatrix}
 		\nabla{L}_\K(x_\K,  \bw_\K,{\uptheta}_\K)\\
 		h_\K(x_\K, \bw_\K)
 	\end{bmatrix} \perp\begin{bmatrix}
 		\bw_\K\\
 		{\uptheta}_\K
 	\end{bmatrix} \geq 0.\label{eq:pCompliment}
 \end{align}  
 Under Assumption \ref{ass:convex} and Lemma \ref{lem:nLCP}, this problem has a unique solution if one exists, making $\texttt{SOL}(\mathrm{pCP}):\mathbb{R}^{nK}\rightarrow \mathbb{R}_{+}^{(c+s)K}$ a piecewise single-valued mapping. Now, consider the composite map  
 \begin{align} 
 	\texttt{SOL}(\mathrm{pCP})\circ \mathcal X_\text{pFS}:\mathbb R_+^{(s+c)K}\rightarrow \mathbb R_+^{(s+c)K}. \label{eq:fixedpointpFS}
 \end{align}  
 Its fixed points are given by  
 \begin{align}
 	&\mathsf Q:=\Big\{({\bw_{\K},{\uptheta}_{\K}})\in \mathbb R^{(s+c)K}_+ ~\Big|~ (\bw_k,\boldsymbol \uptheta_k)=\texttt{SOL}(\mathrm{pCP}_k(x_k))\neq\emptyset, \notag \\
 	&\hspace{0.5in} k\in \K_r,~x_0\in \mathbb R^n \text{ (given)}, \notag \\
 	&\hspace{0.5in} x_{k+1}=f_k\big(x_k,\xi_k^{1\star}(x_k;\{(\bw_{\tau}, {\uptheta}_{\tau})\}_{\tau=k+1}^K),\notag \\
 	&\hspace{1in}
 	\xi_k^{2\star}(x_k;\{(\bw_{\tau}, {\uptheta}_{\tau})\}_{\tau=k+1}^K)\big),~k\in \K_l
 	\Big\}.\label{eq:Qset}
 \end{align}  
 The set $\mathsf{Q}$ in \eqref{eq:Qset} is precisely the set of parameters for which the relations in Theorem \ref{thm:equivtheorem} hold. This observation is summarized in the following proposition.
\begin{figure} 
	\centering 
	\includegraphics[scale=0.425]{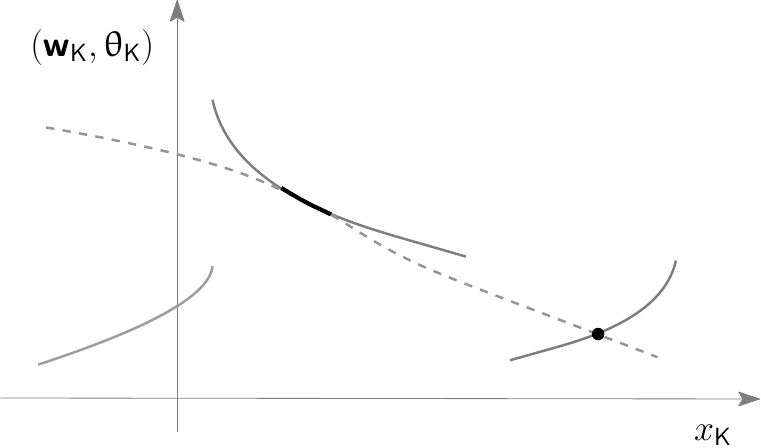}
	\caption{Illustration of set $\mathsf{Q}$  defined in \eqref{eq:Qset} as the intersection of the continuous   map \eqref{eq:Xvectoreq} (gray dashed curve) and piecewise single valued map \eqref{eq:pCompliment} (gray normal curve).}
	\label{fig:solset}
\end{figure} 
 \begin{proposition}\label{prop:fixedpoint}
       Consider the CNZDG described by \eqref{eq:state}-\eqref{eq:objective} and the pNZDG described by \eqref{eq:pnzdg}. Let the assumptions stated in Theorem \ref{thm:equivtheorem} hold. Assume $\texttt{SOL}(\mathrm{pCP}_{0}(x_{0})\neq \emptyset $ and $\mathsf{Q}\neq \emptyset ,$ then the FSN strategies of the players are given by $\{\psi ^{i\star }\equiv (\{(\gamma_{k}^{i\star }(x_{k}),v_{k}^{i\star })\}_{k\in \mathsf{K}_{l}\}},v_{K}^{i\star }),~i\in \mathsf{N}\}$ where the simultaneous decisions $(v_{k}^{1\star },v_{k}^{2\star })$ satisfy $(\mathsf{v}_{0}^{\star },\boldsymbol{\upmu}_{0}^{\star })\in \texttt{SOL}(\mathrm{pCP}_{0}(x_{0}))$ and $(\mathsf{v}_{K}^{\star},\boldsymbol{\upmu}_{K}^{\star })\in \mathsf{Q}$, and the sequential decisions $(\gamma _{k}^{1\star },\gamma _{k}^{2\star}) $ satisfy $\gamma _{k}^{i\star }(x_{k})=\xi _{k}^{i\star }(x_{k};\{(\mathsf{v}_{\tau }^{\star },{\upmu}_{\tau }^{\star })\}_{\tau=k+1}^{K}),~k\in \mathsf{K}_{l}$.
\end{proposition}

\begin{proof} The proof follows from the proof of Theorem \ref{thm:equivtheorem} and the parameter set \eqref{eq:Qset}.
\end{proof}
\begin{remark} \label{rem:multiplicity}
%
From Proposition \ref{prop:fixedpoint}, the FSN solutions correspond to the intersection points of the parametric maps \eqref{eq:Xvectoreq} and \eqref{eq:pCompliment} (see Figure \ref{fig:solset}). This implies that there may exist none (when \( \mathsf{Q} = \emptyset \)), a unique solution, multiple solutions, or even a continuum of solutions (when \( \mathsf{Q} \neq \emptyset \)). In the case of multiple solutions, each FSN state trajectory admits a unique simultaneous equilibrium (by Assumption \ref{ass:convex}) at every time instant. Since the follower’s response is also unique for each leader’s sequential decision \( \gamma_k^{1\star} \), the leader can locally enforce her FSN sequential decision at each time \( k \). Moreover, following Remark \ref{rem:FSNodering}, the leader can order the FSN solutions by her global cost, resulting in admissible FSN solutions.
\end{remark}
\begin{remark}
 Following Theorem \ref{thm:equivtheorem} and Proposition \ref{prop:fixedpoint} the FSN solution is obtained using the parameter set $\mathsf{Q}$ and the pFS solution of the associated unconstrained pNZDG. 
	In particular, the FSN strategies and value functions of the players are obtained as 
	\eqref{eq:equivofsol} and \eqref{eq:WequivWp} respectively. 
\end{remark} 
  \tb 
  	{Although equations \eqref{eq:valfunpar} and \eqref{eq:Qset} characterize the value functions \( W_{p}^{i}(k,\cdot) \) and the set \( \mathsf{Q} \), deriving their exact functional forms for general CNZDGs remains a significant challenge, particularly from a computational standpoint. Even in unconstrained difference games, where feedback Nash and Stackelberg equilibria have been characterized for general nonlinear cases, the literature predominantly relies on linear-quadratic models due to their computational tractability; see \cite{Basar:99} and \cite{Engwerda:05}. Motivated by this, in the next section we demonstrate that for linear-quadratic difference games, the set \( \mathsf{Q} \) can be obtained by solving a linear complementarity problem. Moreover, the parametric value functions \( W_{p}^{i}(k,\cdot) \) can be shown to be quadratic in the state, enabling the determination of FSN solutions.}
  
 \section{Linear-quadratic case}
 \label{sec:CLQDG}
 In this section, we specialize the results of the previous section to a linear-quadratic setting and present a method for computing the FSN solution. To this end, we consider a discrete-time, finite-horizon linear-quadratic difference game with mixed affine coupled inequality constraints (CLQDG). The state variable evolves according to the discrete-time linear dynamics:  
 \begin{subequations}\label{eq:CLQDG}
 \begin{align}
 	x_{k+1} = A_k x_k+B^1_k u_k^1+B^2_k u_k^2,\quad k \in \K_l,\label{eq:LQstate}
 \end{align}
 with a given initial state $x_0 \in \R^n$, where $A_k \in \R^{n\times n}, B_k^{i}\in \R^{n\times m_i}$, $u^i_k \in \U^i_k$ for $i \in \{1,2\}$ and $\U^i_k \subset \R^{m_i}$ represents the admissible action set of player $i$. The mixed coupled constraints \eqref{eq:constraints} for player $ i \in \{1,2\}$ are given by  
 \begin{align}
 	M_k^ix_k+N_k^i\bv_k+r_k^i \geq 0, \, v_k^i\geq 0, ~ k \in \K, \label{eq:LQconstraints}
 \end{align}
 where $M_k^{i} \in \R^{c_i\times n}$, $N_k^{i}\in \R^{c_i\times s}$ and $r_k^{i}\in \R^{c_i}$. The terminal and instantaneous quadratic costs for player $i\in \{1,2\}$ are given by 
 	\begin{align}
 		g_{K}^i(x_K, \mathsf{v}_K)  &=\tfrac{1}{2} x_{K}^\prime Q^i_{K}x_{K}+{p_K^{i}}^\prime x_K+ \tfrac{1}{2} \mathsf{v}_K^\prime {D}^{i}_K \mathsf{v}_K\notag\\&\quad
 		+x_K^\prime {L}^{i}_K \mathsf{v}_K+{d_K^{i}}^\prime\mathsf{v}_K,\label{eq:LQcostterm}\\
	g_k^i(x_k, \mathsf{u}_k, \mathsf{v}_k)  
 		&=\tfrac{1}{2} x_k^\prime Q^i_k x_k+ {p_k^{i}}^\prime x_k+\tfrac{1}{2} \sum_{j=1}^{2}{u_k^{j}}^\prime R^{i j}_k u_k^j
 		+\tfrac{1}{2} \mathsf{v}_k^\prime {D}^{i}_k \mathsf{v}_k\notag\\&\quad+ x_k^\prime {L}^{i}_k \mathsf{v}_k+{d_k^{i}}^\prime\mathsf{v}_k ,\label{eq:LQcost}
 	\end{align}
 \label{eq:LQCNZDG}%
 \end{subequations}
 where $R^{i j}_{k}\in \R^{m_i \times m_j}$ for  $k \in \K_l$, and $Q^i_{k}\in \R^{n \times n}$, $Q^i_{k} =Q^{i'}_{k}$, $p^i_{k}\in \R^{n}$, $D^i_{k}\in \R^{s \times s}$,  $L^i_{k}\in \R^{n \times s}$, $d^i_{k}\in \R^{s}$ for $k \in \K$.  Similar to Assumption \ref{ass:A1}, we have the following assumptions for CLQDG.   
 \begin{assumption}\label{ass:CLQDG}
 	\begin{enumerate}[label = (\roman*)]
 		\item \label{item:1} The joint admissible action sets $\U_k(x_k) \subset \R^{m}$ for $k \in \K_l$ are such that the joint action sets $ 
 		\mathsf V(x_k)$, defined in \eqref{eq:Vadmissible}, are nonempty and bounded for all 
 		for all $k \in \K$. 
 		\item \label{item:2} The matrices $\{[N_k^i]_i,~k \in \K,~i=1,2\}$ have full rank.
 		\item \label{item:3} The matrix $ {D}_k+ {D}_k^{\prime}$ is positive definite for all $k\in\K$, where ${D}_k=\begin{bmatrix}
 			[ {D}^{1}_k]_{11}          &[ {D}^{1}_k]_{12} \\
 			[ {D}^{2}_k]_{21}          &[ {D}^{2}_k]_{22}
 		\end{bmatrix}$.
 	\end{enumerate} 
 \end{assumption}
Assumption \ref{ass:CLQDG}.\ref{item:2} ensures that the constraint qualification conditions hold. Assumption \ref{ass:CLQDG}.\ref{item:3} provides a sufficient condition for the cost functions \eqref{eq:LQcostterm}–\eqref{eq:LQcost} to be strictly diagonally convex in the decision variables \(\mathsf{v}_k\) at each stage \(k \in \mathsf{K}\); see also Assumption \ref{ass:convex} for CNZDG. 
\begin{remark} \label{rem:diffCLQDG}
The game model (CLQDG) analyzed in this section is similar to that in \cite{Reddy:15,Reddy:17} but differs in strategic interactions, information structure, and the solution concept. While \cite{Reddy:15,Reddy:17} examine conditions for the existence of open-loop and feedback Nash equilibria in CLQDG, this paper focuses on the existence conditions for a feedback Stackelberg-Nash solution under quasi-hierarchical interactions.  
\end{remark}
Similar to pNZDG, as defined in the previous section, we use \eqref{eq:LQCNZDG} to define the following parametric unconstrained LQ difference game (pLQDG) associated with CLQDG.  
\begin{subequations}\label{eq:PLQ}
	\begin{align}
     &\mathrm{pLQDG}:
		~ \min_{\tilde{\bu}^i}\Bigl\{\bar{J}_i(x_0,\tilde{\bu}^1,\tilde{\bu}^2;\{(\bw_{\tau}, {\uptheta}_{\tau}) \}_{\tau \in \K}) =  	g_{K}^i(x_K, \mathsf{w}_K)\notag\\
		& +\sum_{k\in\K_l}	g_k^i(x_k, \mathsf{u}_k, \mathsf{w}_k) -\sum_{k\in\K}{\theta_k^{i}}^{\prime}(M_k^ix_k+N_k^i\bw_k+r_k^i)\Bigr\},\label{eq:PLQobjective}\\ 
		&\text{subject to} ~x_{k+1} = A_k x_k+B^1_k u_k^1+B^2_k u_k^2,~ k \in \K_l.\label{eq:PLQstate}
	\end{align}
\end{subequations}

For the pLQDG, due to linearity of the state dynamics and quadratic nature of the cost functions, we have the following assumption on pFS solutions and value functions:
\begin{assumption}\label{ass:pFSstrategyandvalue}
	The pFS decisions of the players $\{\xi_k^{1\star},\xi_k^{2\star}\}$ are  affine functions of the state variable
	\begin{align}
		u_k^{i\star}\equiv\xi_k^{i\star}(x_k;\{(\bw_{\tau}, \boldsymbol{\uptheta}_{\tau})\}_{\tau=k+1}^K)=E_k^ix_k+F_k^i,~ i\in \{1,2\},\label{eq:AffineFB_law}
	\end{align}
	where, $E_k^i \in \R^{m_i \times n}$, $F_k^i \in \R^{m_i}$, $k \in \K_l$. Further, the parametric value function for player $i\in \{1,2\}$ at time $k\in \mathsf K$ has the following form:
	\begin{align}
		&W_p^i\big(k, x_k;\{(\bw_{\tau}, \boldsymbol{\uptheta}_{\tau})\}_{\tau=k}^K\big)
		= \tfrac{1}{2}x_k^\prime  S^i_k x_k+{s_k^{i}}^\prime  x_k+m_k^i\notag\\&\quad
		+\sum_{\tau=k}^{K}\Big(\tfrac{1}{2}\begin{bmatrix}
			\bw_\tau\\
			\bm{\uptheta}_\tau
		\end{bmatrix}^\prime \begin{bmatrix}
			D^i_\tau & -{N^{i}_\tau}^\prime \\
			N^{i}_\tau & 0
		\end{bmatrix}\begin{bmatrix}
			\bw_\tau\\
			\bm{\uptheta}_\tau                               \end{bmatrix}
		+{d_\tau^{i}}^\prime \bw_\tau-{r_\tau^{i}}^\prime \theta_\tau^i\Big), \label{eq:PValue}
	\end{align}
	where $S_k^i \in \R^{n \times n}, s_k^i \in \R^n$ and $m_k^i \in \R$.
\end{assumption}
Next, using the follower's quadratic parametric value function \eqref{eq:PValue} at time \( k+1 \), along with the quadratic objective \eqref{eq:LQcost} and linear state dynamics \eqref{eq:LQstate}, the follower’s objective at time \( k \in \K_l \) is written as follows:
\begin{align*}
    &g_k^2(x_k, (u_k^1, u_k^{2}), \bw_k)-{\theta_k^{2}}^\prime (M_k^2x_{k}+N_k^2\bw_{k}+r_k^2)\\&\quad+W_p^2\big(k+1, A_k x_k+B^1_k u_k^1+B^2_k u_k^{2};\{(\bw_{\tau}, \bm{\uptheta}_{\tau})\}_{\tau=k+1}^{K}\big)\nonumber\\
   &=
    \tfrac{1}{2} x_k^\prime Q^2_k x_k+ {p_k^{2}}^\prime x_k+\tfrac{1}{2} \sum_{j=1}^{2}{u_k^{j}}^\prime R^{2 j}_k u_k^j
 		+\tfrac{1}{2} \mathsf{w}_k^\prime {D}^{2}_k \mathsf{w}_k+ x_k^\prime {L}^{2}_k \mathsf{w}_k\\
        &+{d_k^{2}}^\prime\mathsf{w}_k-{\theta_k^{2}}^\prime (M_k^2x_{k}+N_k^2\bw_{k}+r_k^2)+\tfrac{1}{2}(A_k x_k+B^1_k u_k^1+B^2_k u_k^{2})^{'}S^2_{k+1} \nonumber\\
        &\times(A_k x_k+B^1_k u_k^1+B^2_k u_k^{2})+s_{k+1}^{2'}(A_k x_k+B^1_k u_k^1+B^2_k u_k^{2})+m_k^2\nonumber\\
        &+\sum_{\tau=k+1}^{K}\Big(\tfrac{1}{2}\begin{bmatrix}
			\bw_\tau\\
			\bm{\uptheta}_\tau
		\end{bmatrix}^\prime \begin{bmatrix}
			D^2_\tau & -{N^{2}_\tau}^\prime \\
			N^{2}_\tau & 0
		\end{bmatrix}\begin{bmatrix}
			\bw_\tau\\
			\bm{\uptheta}_\tau                               \end{bmatrix}
		+{d_\tau^{2}}^\prime \bw_\tau-{r_\tau^{2}}^\prime \theta_\tau^2\Big).
\end{align*}
 The first-order associated with the follower's optimal response \( u_k^{2\star} \in \bar{\mathsf{R}}_k^2(u_k^1) \) (see \eqref{eq:ParametricRR}) results in the following equation:
\begin{align*}
    &(R^{22}_k+B^{2'}_k S^2_{k+1}B^2_k)u_k^{2\star}+B^{2'}_k S^2_{k+1}(A_k x_k+B^1_k u_k^1)+B^{2'}_k s^2_{k+1}=0.
\end{align*}
If the matrix \( R^{22}_k + {B^{2}_k}^\prime S^2_{k+1} B^2_k \) is positive definite, the follower's objective is strictly convex in \( u_k^2 \) for any given leader action \( u_k^1 \), ensuring a unique optimal response at time \( k \):   
\begin{align}
\bar{\mathsf R}_k^2(u_k^1)&=- \mathrm{\Upsilon}^2_k({B^{2}_k}^\prime  S^2_{k+1}(A_k x_k+B^1_k u_k^1)
	+{B^2_k}^\prime s^2_{k+1}),\label{eq:followerpFS}
\end{align}
 where $\mathrm{\Upsilon}^2_k:=(R^{22}_k+B^{2'}_k S^2_{k+1}B^2_k)^{-1} $. 
 The state equation using follower's (unique) best response from \eqref{eq:followerpFS} can be written as:
\begin{align}
     x_{k+1}&= \Delta_k(A_k x_k+B^1_k u_k^1)-B^2_k \mathrm{\Upsilon}^2_kB^{2\prime}_ks^2_{k+1}. \label{eq:StateRR}
\end{align}
where $\Delta_k:=(I-B^2_k \mathrm{\Upsilon}^2_kB^{2\prime}_kS^2_{k+1})$. 
Using the quadratic parametric value function \eqref{eq:PValue} of the leader at time instant \( k+1 \), along with the quadratic objective \eqref{eq:LQcost}, state dynamics \eqref{eq:StateRR}, and the follower’s unique best response from \eqref{eq:followerpFS}, the leader’s objective at time \( k \in \K_l \) (see \eqref{eq:WPleaderrecursive}) is written as:   
\begin{align}
&g_k^1(x_k, (u_k^1, \bar{\mathsf{R}}_k^2(u_k^1)), \bw_k)-{\theta_k^{1}}^\prime (M_k^1x_{k}+N_k^1\bw_{k}+r_k^1)\notag\\&+W_p^1\big(k+1, x_{k+1};\{(\bw_{\tau}, \bm{\uptheta}_{\tau})\}_{\tau=k+1}^{K}\big)\notag \\
    &=\tfrac{1}{2} x_k^\prime Q^1_k x_k+ {p_k^{1}}^\prime x_k+\tfrac{1}{2} {u_k^{1}}^\prime R^{11}_k u_k^1\nonumber\\
    &+\tfrac{1}{2} ({B^{2}_k}^\prime  S^2_{k+1}(A_k x_k+B^1_k u_k^1)
	+{B^2_k}^\prime s^2_{k+1})^{\prime}{\mathrm{\Upsilon}^2_k}^\prime R^{12}_k \mathrm{\Upsilon}^2_k\nonumber\\
    &\times({B^{2}_k}^\prime  S^2_{k+1}(A_k x_k+B^1_k u_k^1)
	+{B^2_k}^\prime s^2_{k+1})+\tfrac{1}{2} \mathsf{w}_k^\prime {D}^{1}_k \mathsf{w}_k+ x_k^\prime {L}^{1}_k \mathsf{w}_k\nonumber\\
 		&+{d_k^{1}}^\prime\mathsf{w}_k-{\theta_k^{1}}^\prime (M_k^1x_{k}+N_k^1\bw_{k}+r_k^1)+\tfrac{1}{2}\big(\Delta_k(A_k x_k+B^1_k u_k^1)\nonumber\\
        &-B^2_k \mathrm{\Upsilon}^2_kB^{2\prime}_ks^2_{k+1}\big)^{'}S^1_{k+1}\big(\Delta_k(A_k x_k+B^1_k u_k^1)-B^2_k \mathrm{\Upsilon}^2_kB^{2\prime}_ks^2_{k+1}\big)\nonumber\\
        &+s_{k+1}^{1'}\big(\Delta_k(A_k x_k+B^1_k u_k^1)-B^2_k \mathrm{\Upsilon}^2_kB^{2\prime}_ks^2_{k+1}\big)+m_k^1\nonumber \\
    &+\sum_{\tau=k+1}^{K}\Big(\tfrac{1}{2}\begin{bmatrix}
			\bw_\tau\\
			\bm{\uptheta}_\tau
		\end{bmatrix}^\prime \begin{bmatrix}
			D^1_\tau & -{N^{1}_\tau}^\prime \\
			N^{1}_\tau & 0
		\end{bmatrix}\begin{bmatrix}
			\bw_\tau\\
			\bm{\uptheta}_\tau                               \end{bmatrix}
		+{d_\tau^{1}}^\prime \bw_\tau-{r_\tau^{1}}^\prime \theta_\tau^1\Big).\nonumber
\end{align}
 The first-order condition of the leader's optimization problem (see \eqref{eq:WPleaderrecursive}) results in the following equation: 
\begin{align}
	&\big(R^{11}_k+{B^{1}_k}^\prime \big({S^{2}_{k+1}}^\prime  B^{2}_k{\mathrm{\Upsilon}_k^2}^\prime R^{12}_k \mathrm{\Upsilon}_k^2 {B^{2}_k}^\prime  S^2_{k+1}+\Delta_k^\prime S^1_{k+1}\Delta_k\big) B^1_k\big)  u_k^{1\star}\notag\\
 &+{B^{1}_k}^\prime  ({S^{2}_{k+1}}^\prime  B^{2}_k{\mathrm{\Upsilon}_k^2}^{\prime } R^{12}_k \mathrm{\Upsilon}_k^2  {B^{2}_k}^\prime S^2_{k+1}+\Delta_k^\prime S^1_{k+1}\Delta_k)A_k x_k\notag\\
	&+{B^{1}_k}^\prime ({S^{2}_{k+1}}^\prime  B^{2}_k{\mathrm{\Upsilon}_k^2}^{\prime } R^{12}_k \mathrm{\Upsilon}_k^2 {B^{2}_k}^\prime -\Delta_k^\prime  S^1_{k+1}B^2_k\mathrm{\Upsilon}_k^2 {B^{2}_k}^\prime ) s^2_{k+1}\notag\\
	&+{B^{1}_k}^\prime \Delta_k^\prime  s^1_{k+1}=0.\label{eq:LQP1prob1}
\end{align}
If the coefficient of $u_k^{1\star}$ in \eqref{eq:LQP1prob1} is positive definite then the leader's
objective is strictly convex in $u_k^1$ resulting in a unique pFS strategy for leader at time $k$. Next, using \eqref{eq:AffineFB_law} from Assumption \ref{ass:pFSstrategyandvalue},  the follower's pFS strategy is obtained from \eqref{eq:followerpFS} as
$E_k^2x_k+F_k^2= -\mathrm{\Upsilon}^2_k ({B^{2}_k}^\prime  S^2_{k+1}(A_k x_k+B^1_k E_k^1 x_k +B_k^1 F_k^1) +{B^2_k}^\prime s^2_{k+1})$.
The leader's pFS strategy is obtained by substituting $u_k^{1*}=E_k^1x_k+F_k^1$ in \eqref{eq:LQP1prob1} and equating the coefficients of $x_k$ on both sides, as the relation has to hold true for an arbitrary $x_k$. The pFS strategies of the players are solved as 
\begin{subequations}\label{eq:EFeq}
	\begin{align}
	&E_k^1
	=\Delta_k^\prime S^1_{k+1}\Delta_k)A_k-\mathrm{\Upsilon}^1_k{B^{1}_k}^\prime ({S^{2}_{k+1}}^\prime B^{2}_k{\mathrm{\Upsilon}^2_k}^{\prime} R^{12}_k \mathrm{\Upsilon}_k^2 {B^{2}_k}^\prime S^2_{k+1},\label{eq:E1explicit}\\
	&F_k^1
	=-\mathrm{\Upsilon}^1_k{B^{1}_k}^\prime(\Delta_k^\prime s^1_{k+1}+({S^{2}_{k+1}}^\prime B^{2}_k{\mathrm{\Upsilon}^2_k}^{\prime} R^{12}_k \mathrm{\Upsilon}_k^2 {B^{2}_k}^\prime\notag \\
	&\hspace{1.5in}-\Delta_k^\prime S^1_{k+1}B^2_k\mathrm{\Upsilon}^2_k{B^{2}_k}^\prime) s^2_{k+1}),\label{eq:F1explicit}\\
	&E_k^2=-\mathrm{\Upsilon}^2_kB^{2'}_k S^2_{k+1}(A_k+B^1_k E_k^1),\label{eq:E2explicit}\\
	&F_k^2=-\mathrm{\Upsilon}^2_kB^{2'}_k( s^2_{k+1}+S^2_{k+1}B^1_k F_k^1),\label{eq:F2explicit} \\
	&\mathrm{\Upsilon}^2_k=(R^{22}_k+B^{2'}_k S^2_{k+1}B^2_k)^{-1},\label{eq:FollowerRRmatices}\\
	&\mathrm{\Upsilon}^1_k= (R^{11}_k+{B^{1}_k}^\prime({S^{2}_{k+1}}^\prime B^{2}_k{\mathrm{\Upsilon}^2_k}^{\prime} R^{12}_k \mathrm{\Upsilon}_k^2 {B^{2}_k}^\prime S^2_{k+1}+\Delta_k^\prime  S^1_{k+1}\Delta_k)B^1_k)^{-1}.
	\label{eq:LeaderRRmatices}
\end{align} 
\end{subequations}
The above steps are summarized in the following theorem.
\begin{theorem}
	Consider the pLQDG described by \eqref{eq:PLQ} with parameters $\{(\bw_{\tau}, \boldsymbol{\uptheta}_{\tau})\}_{\tau=0}^K$, and let Assumption \ref{ass:pFSstrategyandvalue} hold. Define $S_k^i$, $s_k^i$ and $m_k^i$ such that the following backward recurrence
	equations are verified for $i\in \{1,2\}$ 
\begin{subequations}\label{eq:PlayerSsmEq}
	\begin{align}
		S^i_{k} &= Q^i_k+\sum_{j=1}^{2}{E_k^{j}}^\prime R^{ij}_k E_k^j+(A_k+\sum_{j=1}^{2}B^j_k E_k^j)^\prime S^i_{k+1}(A_k+\sum_{j=1}^{2}B^j_k E_k^j),\label{eq:Si}
		\end{align} \begin{align} 
		s^i_{k} &= p_k^{i}+L^{i}_k \bw_k-{M_k^i}^\prime\theta_k^{i}+\sum_{j=1}^{2}{E_k^{j}}^\prime R^{ij}_k F_k^j+(A_k+\sum_{j=1}^{2}B^j_k E_k^j)^\prime\notag\\
        &\qquad \times(s_{k+1}^{i}+S^i_{k+1}\sum_{j=1}^{2}B^j_k F_k^j),\label{eq:si} \\
		m_{k}^i &= m_{k+1}^i+\tfrac{1}{2}\sum_{j=1}^{2}{F_k^{j}}^\prime R^{ij}_k F_k^j+\tfrac{1}{2}(\sum_{j=1}^{2}B^j_k F_k^j)^\prime S^i_{k+1}(\sum_{j=1}^{2}B^j_k F_k^j)\notag\\
        &\qquad +(\sum_{j=1}^{2}B^j_k F_k^j)^\prime s^i_{k+1}.\label{eq:mi}
	\end{align}
\end{subequations}
with terminal conditions  $
	S^i_{K} = Q^i_K$,  $s^i_{K} = p_K^{i}+L^{i}_K \bw_K-{M_K^{i}}^\prime \theta_K^{i}$, and  $m_{K}^i = 0$. If the matrices $\mathrm{\Upsilon}^1_k$ and $\mathrm{\Upsilon}^2_k$ defined in \eqref{eq:FollowerRRmatices} and \eqref{eq:LeaderRRmatices} are positive definite, then
	$\xi_k^{i\star}(x_k)=E_k^ix_k+F_k^i$ is a pFS solution for pLQDG where $E_k^i$ and $F_k^i$ are given by  \eqref{eq:E1explicit}-\eqref{eq:F2explicit} for $i\in \{1,2\}$,~$k\in \K_l$. The pFS state trajectory is given by
	\begin{align}\label{eq:pFSstatetraj}
		x_{k+1}=(A_k+\sum_{i=1}^2 B_k^i E_k^i)  x_k + \sum_{i=1}^2 B_k^i F_k^i.
	\end{align}
\end{theorem} 
 \begin{proof}
    For both leader and follower, under Assumption \ref{ass:pFSstrategyandvalue}, the backward recursive relations \eqref{eq:PlayerSsmEq} follows by comparing the coefficients of state in the verification result \eqref{eq:WPleaderrecursive1} of Lemma \ref{lem:VerificationTh2}.  The positive definiteness of the matrices \(\mathrm{\Upsilon}^2_k\) for all \(k \in \K_l\) ensures a unique optimal response for the follower at each time instant. Similarly, the positive definiteness of \(\mathrm{\Upsilon}^1_k\) for all \(k \in \K_l\) guarantees a unique parametric Stackelberg solution for the leader at every time instant.  
  \end{proof} 
\subsection{FSN solution as a linear complementarity problem}\label{sec:LQFSN} 
In this subsection, we outline an approach to obtain the elements of the set $\mathsf{Q}$, as defined in \eqref{eq:Qset}, by solving a large-scale linear complementarity problem. The procedure follows a sequence of steps similar to those in \cite{Reddy:17,Reddy:19} and involves eliminating the state variables in \eqref{eq:Qset} using \eqref{eq:pFSstatetraj}. To this end, we first introduce some notation.   Denote
  $ \mathsf{p}_{k}:= \col{p^{1}_{k}, p^{2}_{k}}$,   $\mathsf {s}_{k}:=\col{s^{1}_{k}, s^{2}_{k}}$, $\mathsf B_k:=\row{B_k^1,B_k^2}$, $\mathsf{E}_{k}:=\col{E^1_k, E^2_k}$,  $ \mathsf{F}_{k}:=\col{F^1_k, F^2_k}$, $\mathsf{L}_{k}:=\col{L^1_{k}~L^2_{k}}$, $\mathsf{M}_{k}:= M^1_k \oplus M_k^2$, and the gain matrices $G_{k+1}$ and $H_{k+1}$ partitioned as
 \begin{align*}
 	&[ {G}_{k+1}]_{11}:=-\mathrm{\Upsilon}_k^1{B^{1}_k}^\prime\Delta_k^\prime,\quad\\
 	&[ {G}_{k+1}]_{12}:=-\mathrm{\Upsilon}_k^1 {B^{1}_k}^\prime({S^{2}_{k+1}}^\prime B^{2}_k{\mathrm{\Upsilon}^2_k}^{\prime} R^{12}_k \mathrm{\Upsilon}_k^2 {B^{2}_k}^\prime-\Delta_k^\prime S^1_{k+1}B^2_k\mathrm{\Upsilon}_k^2{B^{2}_k}^\prime), 	\\
 	&[ {G}_{k+1}]_{21}:=-\mathrm{\Upsilon}_k^2 {B^2_k}^\prime S^2_{k+1}B^1_k[ {G}_{k+1}]_{11}\\
 	&[ {G}_{k+1}]_{22}:= -\mathrm{\Upsilon}_k^2  {B^2_k}^\prime(I+S^2_{k+1}B^1_k [ {G}_{k+1}]_{12}),\\ 
&[H_{k+1}]_{ii}:=\mathrm{\Upsilon}_k^{ii} [G_{k+1}]_{ii}
 +(A_k+\mathsf{B}_k \mathsf{E}_k)^{'}+\mathrm{\Upsilon}_k^{ij} [G_{k+1}]_{ji},\\
 &[H_{k+1}]_{ij}:=\mathrm{\Upsilon}_k^{ii} [G_{k+1}]_{ij}+\mathrm{\Upsilon}_k^{ij} [G_{k+1}]_{jj},
\end{align*}
where $\mathrm{\Upsilon}_k^{ii}= E_k^{i'}R^{ii}_k+(A_k+\mathsf{B}_k \mathsf{E}_k)^{'}S^i_{k+1}B^i_k $ and
$\mathrm{\Upsilon}_k^{ij}= E_k^{j'}R^{ij}_k+(A_k+\mathsf{B}_k \mathsf{E}_k)^{'}S^i_{k+1}B^j_k $
for $i,j\in \{1,2\},~i\neq j$.
	Using the above, \eqref{eq:F1explicit}-\eqref{eq:F2explicit} and \eqref{eq:si} can be written in vector form as
\begin{align}
	\mathsf{F}_{k}& = {G}_{k+1} \mathsf{s}_{k+1}, \label{eq:Fequation}\\
	 \mathsf{s}_{k}&= \mathsf{p}_{k}+[ \mathsf{L}_{k} ~- \mathsf{M}^\prime_{k}][
	 \mathsf w_k^\prime ~~
	{\uptheta}_{k}^\prime]^\prime+{H}_{k+1} \mathsf{s}_{k+1},\label{eq:sEquation}
\end{align}
with the terminal condition given by $
	 \mathsf{s}_{K}= \mathsf{p}_{K}+[\mathsf{L}_{K} ~- \mathsf{M}^\prime _{K}][
	\bw_{K}^\prime ~~
	{\uptheta}_{K}^\prime ]^\prime$. Next, \eqref{eq:sEquation} can be solved as
\begin{align}
	\mathsf {s}_{k}=\sum_{\tau =k}^{K}\varphi(k, \tau)\big( \mathsf{p}_{\tau}+[ \mathsf{L}_{\tau} ~- \mathsf{M}^\prime _{\tau}][
	\bw_{\tau}^\prime ~~
	{\uptheta}_{\tau}^\prime ]^\prime \big),\label{eq:backwarddiff}
\end{align}
where  the associated state transition matrices $\varphi(k, \tau)$ are given by $\varphi(k, \tau)= I$ for $\tau=k$, and $\varphi(k, \tau)=
	 {H}_{k+1} {H}_{k+2}\cdots {H}_{\tau}$ for $\tau>k$.
Using \eqref{eq:Fequation} in \eqref{eq:pFSstatetraj}, the pFS state variable evolves according the forward linear difference equation:
\begin{align*}
    x_{k+1} = \bar{A}_k x_k+\bar{B}_k \mathsf{s}_{k+1},
\end{align*}
with  $\bar{A}_k:=A_k+\mathsf  B_k \mathsf  E_k$ and $\bar{B}_k:= \mathsf B_k {G}_{k+1}$. The solution of this linear forward difference equation for $k \in \K_r$ is given by
\begin{align}
	x_k=\phi(0,k)x_0+\sum_{\rho=0}^{k-1}\phi(\rho+1,k)\bar{B}_\rho \mathsf{s}_{\rho+1},\label{eq:forwarddiff}
\end{align}
where the associated state transition matrices $\phi(\rho, k)$ are defined as $\phi(\rho, k)= 
	 {I}$ for $\rho=k$, and $\phi(\rho, k)=	\bar{A}_{k-1}\bar{A}_{k-2}\cdots\bar{A}_\rho$ for $\rho<k$. Using \eqref{eq:backwarddiff} in \eqref{eq:forwarddiff}, the pFS state trajectory for $k \in \K_r$ is given as follows
\begin{align}
	x_k 
	&=\phi(0,k)x_0+\sum_{\tau =1}^{K}\Big(\sum_{\rho=1}^{\mathrm{min}\,\,(k,\tau)}\phi(\rho,k)\bar{B}_{\rho-1}\varphi(\rho, \tau)\Big)\notag\\
 &\hspace{1.25in}\times\big( {p}_{\tau}+[ \mathsf{L}_{\tau} ~- \mathsf{M}^\prime _{\tau}][
	\mathsf w_{\tau}^\prime ~~
	{\uptheta}_{\tau}^\prime ]^\prime \big).\label{eq:finalforwarddiff}   
\end{align}
Aggregating the variables in \eqref{eq:finalforwarddiff}  by  $ \mathsf{p}_{\K}:=\col{ \mathsf{p}_{k}}_{k=1}^K,~ x_{\K}:=\col{x_{k}}_{k=1}^K,~\bw_{\K}:=\col{\bw_{k}}_{k=1}^K$ and $ {\uptheta}_{\K}:=\col{ {\uptheta}_{k}}_{k=1}^K$, the pFS state trajectory $x_k$, $k\in \K_r$ is  written compactly as  
\begin{align}
	x_{\K}= {\Phi}_0x_0+ {\Phi}_1 \mathsf{p}_{\K}+ {\Phi}_2\bw_{\K}+ {\Phi}_3{\uptheta}_{\K},\label{eq:aggregateState}
\end{align}
where for all $ k,\tau \in \K_r$ the matrices appearing on the right-hand side of \eqref{eq:aggregateState} are given by
    $[ {\Phi}_0]_{k}:=\phi(0,k)$,  $[ {\Phi}_1]_{k \tau}:=\sum_{\rho=1}^{\mathrm{min}(k,\tau)}\phi(\rho,k)\bar{B}_{\rho-1}\varphi(\rho, \tau)$, $[ {\Phi}_2]_{k \tau}:=\sum_{\rho=1}^{\mathrm{min}(k,\tau)}\phi(\rho,k)\bar{B}_{\rho-1}\varphi(\rho, \tau) \mathsf{L}_{\tau}$,  $[ {\Phi}_3]_{k \tau}:=-\sum_{\rho=1}^{\mathrm{min}\,\,(k,\tau)}\phi(\rho,k)\bar{B}_{\rho-1}\varphi(\rho, \tau) \mathsf{M}^\prime _{\tau}$. Next,  we compute the following terms:
    \begin{align*}
		\nabla_{v_k^i} g^i_{v,k}(x_k, \mathsf{w}_{k})&=\nabla_{v_k^i} g^i_{v,k}(x_k, \bv_k)|_{\bv_k=\mathsf{w}_{k}}\\
        &= \big[[{D}^{i}_k]_{i1} ~[{D}^{i}_k]_{i2}\big] \mathsf{w}_{k}+ [{L}^{i}_k]_i^{\prime}x_k+[d_k^{i}]_{i},\\
		\nabla_{v_k^i} h_k^i(x_k, \mathsf{w}_{k})&=\nabla_{v_k^i} h_k^i(x_k, \bv_k)|_{\bv_k=\mathsf{w}_{k}}=[N_k^i]_i,\quad i=1,2.
	\end{align*}
    Using the above expressions, the terms $\nabla {L}_{k}(x_{k},\mathsf{w}_{k},{\uptheta}_{k})$ and $h_{k}(x_{k},\mathsf{w}_{k})$ are simplified as follows:
    \begin{align}
		&\nabla {L}_{k}(x_{k},\mathsf{w}_{k},{\uptheta}_{k})=\col{\nabla_{v_k^i} g^i_{v,k}(x_k, \mathsf{w}_{k})-{\theta_k^{i}}^{\prime}\nabla_{v_k^i} h_k^i(x_k, \mathsf{w}_{k})}_{i=1}^{2}\nonumber\\
		&=\begin{bmatrix}
			[{D}^{1}_k]_{11}          &[{D}^{1}_k]_{12}  &-{[N_k^1]_1}^{\prime}  &~~\mathbf{0}\\
			[{D}^{2}_k]_{21}          &[{D}^{2}_k]_{22}  &~~\mathbf{0}     &-{[N_k^2]_2}^{\prime}
		\end{bmatrix}\begin{bmatrix}
			\mathsf{w}_{k}\\
			{\uptheta}_{k}
		\end{bmatrix}\nonumber\\
        &\hspace{1.75in}+\begin{bmatrix}
			[{L}^{1}_k]_1^{\prime}\\
			[{L}^{2}_k]_2^{\prime}
		\end{bmatrix}x_k+\begin{bmatrix}
			[d_k^{1}]_{1}\\
			[d_k^{2}]_{2}
		\end{bmatrix}\nonumber\\
		&=\big[
			{D}_k  ~-{([N_k^1]_1\oplus [N_k^2]_2)}^{\prime}
		\big]\begin{bmatrix}
			\mathsf{w}_{k}\\
			{\uptheta}_{k}
		\end{bmatrix}+\big[[{L}^{1}_k]_1~ [{L}^{2}_k]_2\big]^{\prime}x_k+\begin{bmatrix}
			[d_k^{1}]_{1}\\
			[d_k^{2}]_{2}
		\end{bmatrix},\nonumber
        \end{align}
        \begin{align}
        &h_k(x_k, \mathsf{w}_{k})=\col{h_k^i(x_k, \mathsf{w}_{k})}_{i=1}^{2}=\begin{bmatrix}
			M_k^1x_k+N_k^1\mathsf{w}_{k}+r_k^1\\
			M_k^2x_k+N_k^2\mathsf{w}_{k}+r_k^2
		\end{bmatrix}\nonumber\\
		&=\begin{bmatrix}
			\col{N_k^1,N_k^2} &\mathbf{0}
		\end{bmatrix}\begin{bmatrix}
			\mathsf{w}_{k}\\
			{\uptheta}_{k}
		\end{bmatrix}+\begin{bmatrix}
			M_k^1\\
                M_k^2
		\end{bmatrix}x_k+\begin{bmatrix}
			r_k^1\\
                r_k^2
		\end{bmatrix}.\nonumber
	\end{align}    
    Next, aggregating $\nabla {L}_{k}(x_{k},\mathsf{w}_{k},{\uptheta}_{k})$ and $h_{k}(x_{k},\mathsf{w}_{k})$ for all \( k \in \K \), we obtain the following equations:
    \begin{align*}
        &\nabla{L}_\K(x_\K,  \bw_\K,{\uptheta}_\K)=\mathrm{col}(\nabla {L}_{k}(x_{k},\mathsf{w}_{k},{\uptheta}_{k}))_{k=1}^{K}\\
        &\hspace{0.75in}=\begin{bmatrix}
			\mathsf {D}_\K & - {\bar{\mathsf{N}}_\K}^\prime 
		\end{bmatrix}\begin{bmatrix}
			\bw_\K\\
			{\uptheta}_\K
		\end{bmatrix}
		+\begin{bmatrix}
			{ \bar{ \mathsf L}_\K}^\prime 
		\end{bmatrix}x_\K+\begin{bmatrix}
			\mathsf{d}_\K
		\end{bmatrix},\\
        &h_{\mathsf{K}}(x_{\mathsf{K}},\mathsf{w}_{\mathsf{K}})=\mathrm{col}(h_{k}(x_{k},\mathsf{w}_{k}))_{k=1}^{K}\\
         &\hspace{0.75in}=\begin{bmatrix}
			\mathsf {N}_\K & \mathbf{0}
		\end{bmatrix}\begin{bmatrix}
			\bw_\K\\
			{\uptheta}_\K
		\end{bmatrix}
		+\begin{bmatrix}
			{\bar{\mathsf M}}_\K
		\end{bmatrix}x_\K+\begin{bmatrix}
			\mathsf{r}_\K
		\end{bmatrix},
    \end{align*}
    where $\mathsf{D}_\K=\oplus_{k=1}^{K}D_k$, $\bar{\mathsf{N}}_\K=\oplus_{k=1}^{K} ([N_k^1]_1\oplus [N_k^2]_2)$, $\mathsf{N}_\K=\oplus_{k=1}^{K} (\col{N_k^1, N_k^2})$,  ${\bar{\mathsf L}}_\K=\oplus_{k=1}^{K} (\row{[L_k^1]_1,[L_k^2]_2})$, 
${\bar{\mathsf M}}_\K=\oplus_{k=1}^{K} (\col{M_k^1,M_k^2})$,  $\mathsf{d}_\K=\col{ \col{[d_k^1]_1,[d_k^2]_2}}_{k=1}^K$, $\mathsf {r}_\K=\col{\col{r_k^1,r_k^2}}_{k=1}^K$. Finally, using the above expressions, the parametric complementarity problem   \eqref{eq:pCompliment} can be written as the following parametric linear complementarity problem ($\mathrm{pLCP}(x_\K)$):
\begin{align}
	0 \leq\begin{bmatrix}
		\mathsf {D}_\K & - {\bar{\mathsf{N}}_\K}^\prime \\
		\mathsf {N}_\K & \mathbf{0}
	\end{bmatrix}\begin{bmatrix}
		\bw_\K\\
		{\uptheta}_\K
	\end{bmatrix}
	+\begin{bmatrix}
		{ \bar{ \mathsf L}_\K}^\prime \\
		{\bar{\mathsf M}}_\K
	\end{bmatrix}x_\K
	+\begin{bmatrix}
		 \mathsf{d}_\K\\
		 \mathsf{r}_\K
	\end{bmatrix} \perp\begin{bmatrix}
		\bw_\K\\
		{\uptheta}_\K
	\end{bmatrix} \geq 0,\label{eq:LQpLCP}  
\end{align}
   Using \eqref{eq:aggregateState} in \eqref{eq:LQpLCP}, we obtain the following single linear complementary problem:
\begin{align}\label{eq:LQMmap}
	\textrm{LCP}: 0 \leq {\mathbf{M}}_{\K}\begin{bmatrix}
		\bw_\K\\
		{\uptheta}_\K
	\end{bmatrix}+ {\mathbf{q}}_{\K} \perp\begin{bmatrix}
		\bw_\K\\
		{\uptheta}_\K
	\end{bmatrix} \geq 0,
\end{align}
where $
	 {\mathbf{M}}_{\K}:=\begin{bmatrix}
		 \mathsf{D}_\K+ {\bar{\mathsf L}_\K}^\prime  {\Phi}_2 & - {\bar{\mathsf{N}}_\K}^\prime+ {\bar{\mathsf L}_\K}^\prime  {\Phi}_3\\
		 \mathsf{N}_\K+ { \bar{\mathsf M}}_\K {\Phi}_2 &  { \bar{\mathsf M}}_\K {\Phi}_3
	\end{bmatrix}$ and
	 ${\mathbf{q}}_{\K}:=\begin{bmatrix}
		\mathsf {d}_\K+{\bar{\mathsf L}_\K}^\prime  {\Phi}_1 \mathsf{p}_\K+ {\bar{\mathsf L}_\K}^\prime {\Phi}_0x_0\\
		 \mathsf{r}_\K+ { \bar{\mathsf M}}_\K {\Phi}_1 \mathsf{p}_\K+ {\bar{\mathsf M}}_\K {\Phi}_0x_0
	\end{bmatrix}$.
\begin{theorem} \label{th:LCP}
  Consider the CLQDG described by \eqref{eq:CLQDG}. Let Assumptions \ref{ass:CLQDG} and \ref{ass:pFSstrategyandvalue} hold true and the matrices $\mathrm{\Upsilon}^1_k$ and $\mathrm{\Upsilon}^2_k$ defined in \eqref{eq:FollowerRRmatices} and \eqref{eq:LeaderRRmatices} are positive definite for all $k\in\K_l$. If $\texttt{SOL}(\mathrm{pLCP}_0(x_0)\neq \emptyset$ and $\texttt{SOL}(\mathrm{LCP})\neq \emptyset$, then the FSN strategies of the players in CLQDG are given by
  \begin{align*}
      \{\psi^{i\star}\equiv (\{(\gamma_k^{i\star}(x_k),v_k^{i\star})\}_{k \in \K_l},v_K^{i\star}),~ i=1,2\},
  \end{align*}
  where the simultaneous decisions are $(\bv_0^{\star}, \bm{\upmu}_0^{\star})=\texttt{SOL}(\mathrm{pLCP}(x_0))$,  $(\bv_{\K}^{\star}, \bm{\upmu}_{\K}^{\star})=\texttt{SOL}(\mathrm{LCP})$ and the sequential decisions are $\gamma_k^{i\star}(x_k) = E_k^ix_k+F_k^i,\, i \in\{1,2\}$, where $E_k^i$ and $F_k^i,\, i \in \{1,2\}$ are given by \eqref{eq:EFeq} with parameters as $(\bv_{\K}^{\star}, \bm{\upmu}_{\K}^{\star})$. 
\end{theorem}
	\begin{proof}
 	The proof follows from Proposition \ref{prop:fixedpoint} and the steps before the theorem.
	\end{proof}
\begin{remark} 
    The LCP in \eqref{eq:LQMmap} may have no solution, a unique one, multiple, or even a continuum of equilibria. When multiple solutions exist, each yields an FSN strategy in which the leader enforces her sequential decision locally at each stage. As noted in Remark~\ref{rem:multiplicity}, these   solutions can be ordered by the leader’s global cost; see also Remark~\ref{rem:FSNodering}. Existence, uniqueness, and numerical methods for solving LCPs are well-studied in the optimization literature; see \cite{Cottle:09} for details. 
 \end{remark}
 \begin{remark}\label{rem:LQVerification}
The conditions in Theorem \ref{th:LCP} can be verified a priori using the problem data. In particular, constraints \eqref{eq:LQconstraints} are affine, and the non-emptiness of \(\V(x_k)\) in Assumption \ref{ass:CLQDG}.\ref{item:1} can be checked as a convex feasibility problem. Furthermore, Assumptions \ref{ass:CLQDG}.\ref{item:2}–\ref{item:3} follow directly from the problem data. Similarly, the positive definiteness of the matrices \(\mathrm{\Upsilon}^1_k\) and \(\mathrm{\Upsilon}^2_k\), defined in \eqref{eq:FollowerRRmatices} and \eqref{eq:LeaderRRmatices}, can be verified for all \( k \in \mathsf{K}_l \) using the problem data.  
 \end{remark}
 
\section{Numerical Illustration}\label{sec:Numerical}
       In this section, we illustrate our results with a numerical example. The setting is the same as in \cite{Reddy:17}, i.e., two firms compete in the same market in quantities (i.e., \`{a} la Cournot) and invest in process R$\&$D to reduce their unit production costs, and in their production capacity. The difference with \cite{Reddy:17} is that here the firms announce sequentially, and not simultaneously, their investments. Following the motivating example in the introduction, we suppose that player 1 is an international company that acts as leader in the investment variables and player 2 is a local firm that acts as follower. Denote by $v_{k}^{i}$ the quantity produced (output) by firm $i\in \{1,2\}$ at time $k$. The price $P(v_{k}^{1},v_{k}^{2})$ is given by the following affine inverse demand:  
\begin{align*}
		P(v_k^1,v_k^2)=\bar{A}_k-\bar{B}_k(v_k^1+v_k^2), \quad k \in \K,
\end{align*}
  where, $\bar{A}_k$ and $\bar{B}_k$ represent, respectively, the maximum willingness to pay and the price sensitivity to the total quantity at time instant $k$. We assume that the demand increases over time, modeled by $\bar{A}_{k}=\bar{A}_{k-1}(1+\epsilon )$ and $\bar{B}_{k}=\bar{B}_{k-1}/(1+\epsilon)$,  for all $k\in \mathsf{K}_{r}$, with $\epsilon >0$.  
The unit production cost of each firm decreases with its stock of R\&D, denoted \( X_{k}^{i} \), which evolves according to the following difference equation:  
\begin{subequations}\label{eq:NIstate}
	\begin{align}
		X_{k+1}^i = \mu^i X_{k}^i + R_k^i + \lambda^i R_k^j,\quad i \neq j,\quad i,j \in \{1,2\},\label{eq:NIstate1}
	\end{align}
	where \( \mu^i \in (0,1) \) and \( R_{k}^{i} \) denotes the investment in R\&D by firm \( i \) at time \( k \in \mathsf{K}_{l} \). The spillover parameter \( \lambda^{i} \in (0,1) \) represents the portion of firm \( i \)’s R\&D investment that spills over to its rival, implying that knowledge generated by a firm is not fully appropriable. Let \( I_{k}^{i} \) be the investment by firm  $i$ to increase its production capacity, \( Y_{k}^{i} \), which evolves as
	\begin{align}
		Y_{k+1}^i = \delta^i Y_{k}^i + I_k^i,\quad i \in \{1,2\},\label{eq:NIstate2}
	\end{align}
\end{subequations}
where $\delta^i\in(0,1)$, and $(1 - \delta^i)$   denotes the depreciation rate. Each firm's production must be non-negative and is upper-bounded by its production capacity:
\begin{align}
	Y_{k}^i \geq v_{k}^i \geq 0,\quad i \in \{1,2\}.\label{eq:NIconstraint}
\end{align}
Production decisions are made simultaneously. The production cost of firm \( i \) is given by \( h_{i}(X_{k}^{i}, v_{k}^{i}) = (c^{i} - \gamma^{i} X_{k}^{i}) v_{k}^{i} \), where \tb{\( c^{i} \)} is the fixed unit cost and \( \gamma^{i} \) is the cost-learning parameter that captures the speed of unit cost reduction. The R\&D and capacity investment costs are modeled as \( g_{i}(R_{k}^{i}) = (a^{i}/2)(R_{k}^{i})^{2} \) and \( f_{i}(I_{k}^{i}) = (b^{i}/2)(I_{k}^{i})^{2}\), respectively, where \( a^{i} \) and \( b^{i} \) are positive parameters. At the terminal time \( K \), the salvage value is given by  $
S_{i}(X_{K}^{i}, Y_{K}^{i}) = (\alpha_{X}^{i}/2)(X_{K}^{i})^{2} + (\alpha_{Y}^{i}/2)(Y_{K}^{i})^{2}$,
where $\alpha_{X}^{i} < 0$, $\alpha_{Y}^{i} < 0$, $i \in \{1,2\}$. Firm $i$'s objective is given by
\begin{align}
	J^i &= \sum_{k=0}^{K-1} \beta^k \left( f_i(R_k^i) + g_i(I_k^i) \right) + \sum_{k=0}^{K} \beta^k \left( h_i(X_k^i, v_k^i) - P(v_k^1, v_k^2) v_k^i \right) \notag \\
	&\quad + \beta^K S_i(X_K^i, Y_K^i),\label{eq:NIobjective}
\end{align}
where \( 0 < \beta < 1 \) is the common discount factor. Each firm \( i \in \{1,2\} \) minimizes the objective \eqref{eq:NIobjective} subject to the dynamics \eqref{eq:NIstate} and capacity constraint \eqref{eq:NIconstraint}. The dynamic duopoly game defined by \eqref{eq:NIstate}–\eqref{eq:NIobjective} fits the structure of the CLQDG model \eqref{eq:CLQDG}. 
For numerical illustration, we use the following parameter values: $K = 14,~ \bar{A}_{0} = 3.5,~ \bar{B}_{0} = 0.5,~ \epsilon = 0.015,~ \beta = 0.9,~ \lambda^{i} = 0.1,~ \mu^{i} = 0.8,~ \delta^{i} = 0.85,~ \gamma^{i} = 0.2,~ a^{i} = 1,~ b^{i} = 1, c^{i} = 0.5,~ \alpha_{X}^{i} = -0.2,~ \alpha_{Y}^{i} = -0.25,~ X_{0}^{i} = 5,~ Y_{0}^{i} = 4,\quad i \in \{1,2\}$. Following Remark~\ref{rem:LQVerification}, it can be readily verified that the assumptions of Theorem~\ref{th:LCP} are satisfied for the chosen parameters. We use the freely available PATH solver (available at \url{https://pages.cs.wisc.edu/~ferris/path.html}) to solve the LCP \eqref{eq:LQMmap}. For the given problem instance, computation takes approximately 0.4 seconds on a MacBook with an M1 Pro processor (10-core CPU) and 32 GB RAM.

       Figure \ref{fig:Fig3} shows the evolution of the stock of knowledge, production capacity, investments, and quantities produced by both firms under the FSN solution. Figure \ref{fig:fig11} depicts the evolution of production capacity and outputs of the foreign firm (leader) and the local firm (follower). For both firms, outputs are upper-bounded by their respective production capacities, with these constraints active from time period $4$ to $13$. We consider the situation where both firms invest simultaneously (e.g., when both are local firms), and compute the feedback-Nash equilibrium using the approach in \cite{Reddy:17}. Figures \ref{fig:fig12}–\ref{fig:fig32} compare the FSN and feedback-Nash equilibrium solutions under this scenario. Although both firms start with identical initial values for stock of knowledge and production capacity, Figures \ref{fig:fig12} and \ref{fig:fig21} show that, over time, these values are higher for the foreign firm (leader) than for the local firm (follower) and the feedback-Nash case. The quantity produced by the foreign firm is also consistently higher; see Figure \ref{fig:fig22}. A similar pattern appears in the investment levels (Figures \ref{fig:fig31}–\ref{fig:fig32}). In all figures from \ref{fig:fig12} to \ref{fig:fig32}, the feedback-Nash equilibrium trajectories lie between those of the leader and follower, highlighting their intermediate nature.   
       
\begin{figure}[h] 
	\centering 
	\subfloat[]{\includegraphics[scale=0.8]{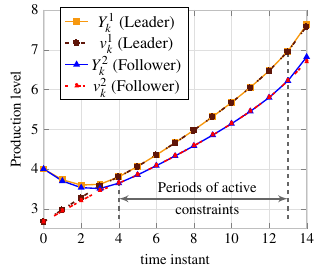} \label{fig:fig11}} 
	\subfloat[]{\includegraphics[scale=0.8]{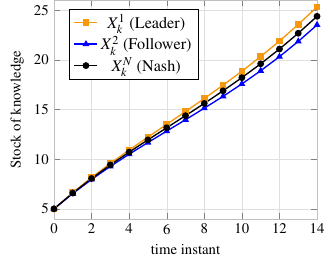} \label{fig:fig12}} \\
	\subfloat[]{\includegraphics[scale=0.8]{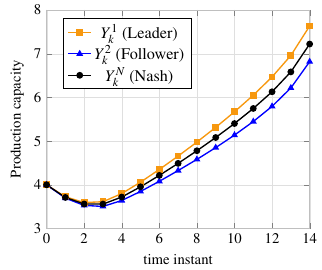} \label{fig:fig21}}  
	\subfloat[]{\includegraphics[scale=0.8]{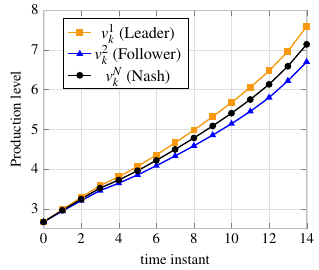} \label{fig:fig22}} \\
	\subfloat[]{\includegraphics[scale=0.8]{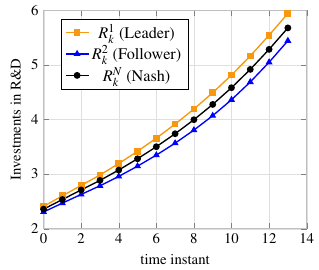} \label{fig:fig31}}  
	\subfloat[]{\includegraphics[scale=0.8]{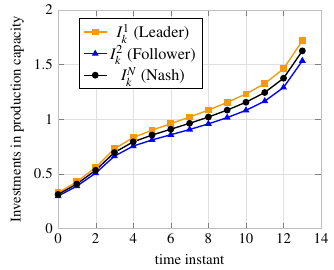} \label{fig:fig32}}
	
	\caption{Evolution of production capacities and production quantities (panel (a)), stock of knowledge (panel (b)), production capacities (panel (c)), quantity produced (panel (d)), investments in R$\&$D (panel (e)) and investments in production capacity (panel (f)).}
	\label{fig:Fig3}
\end{figure} 
\begin{table}[h]
\centering
	\caption {{Comparison of costs with $\lambda^1=\lambda^2$}} \label{tab:2}
	\small{
				\begin{tabular}{cccc} \hline
					 $\lambda^i$   &Leader ($J^1$)        &Follower ($J^2$)         &Feedback Nash ($J^N$)    \\\hline
   0.10          &-35.2466      &-31.4820         &-33.3295        \\ 
  0.15          &-36.5904     &-35.8300        &-36.2220        \\ 
 0.20          &-40.4875      &-39.8770         & -40.1874        \\
   0.25          &-46.3442      &-45.9252         &-46.0684        \\ 
 0.30          &-56.1883      &-56.2889         &-55.8265        \\ 
    0.35          &-76.2580      &-77.9906         &-75.0351        \\
 0.40         &-135.0391      &-143.6161         &-125.9453 \\ \hline
		\end{tabular}}
\end{table}

%
       	In Table \ref{tab:2} we compare players' costs under FSN and feedback-Nash strategies \cite{Reddy:17} for varying spillover parameter ($\lambda^1=\lambda^2$). The leader always benefits from the sequential structure, incurring lower costs than in the feedback-Nash case. For small $\lambda^i$ (top four rows), the follower's costs exceed those of both the leader and feedback-Nash. As $\lambda^i$ increases, the follower gains more from the leader’s R\&D (via \eqref{eq:NIstate1}), reducing her production cost  $h_{2}(X_{k}^{2}, v_{k}^{2})=(c^{2} -\gamma^{2}X_{k}^{2})v_{k}^{2}$. For larger $\lambda^i$ (bottom three rows), the follower’s total cost falls below both the leader’s and feedback-Nash.     	
       	%

\section{Conclusions}\label{sec:Conclusions}
          We studied a class of two-player nonzero-sum difference games with coupled inequality constraints, where players interact sequentially in one decision variable and simultaneously in another. For this quasi-hierarchical interaction, we defined the feedback Stackelberg-Nash (FSN) solution and provided a recursive formulation under a separability assumption on the cost functions. We showed that the FSN solution can be obtained from the parametric feedback Stackelberg solution of an associated unconstrained game with only sequential decisions, using parameter values that satisfy implicit complementarity conditions. In the linear-quadratic case with affine inequality constraints, the FSN solution reduces to a large-scale linear complementarity problem.
          
          \tb{
          	Future directions include extending the model to settings where both sequential and simultaneous decisions influence the state dynamics under mixed constraints, or where the cost includes cross terms between these decisions. In both cases, computing the FSN solution leads to a mathematical program with complementarity constraints (MPCC), which poses challenges due to the nonconvexity and nonsmoothness of the feasible set. Addressing these challenges will require efficient solution methods for MPCCs.}
	\bibliographystyle{ieeetr}
	\bibliography{main.bbl} 
	\begin{IEEEbiography}[{\includegraphics[width=0.9in,clip,keepaspectratio]{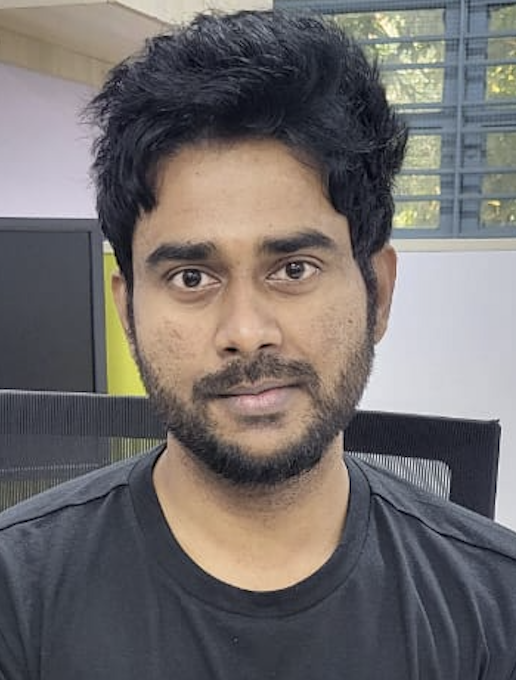}}]{Partha Sarathi Mohapatra} received the B.Tech. degree in Electrical Engineering from Indira Gandhi Institute of Technology, Sarang, India, in 2012. He is currently pursuing the Ph.D. degree with the Department of Electrical Engineering, IIT Madras, Chennai, India. His research interests include game theory, reinforcement learning and stochastic control. 
\end{IEEEbiography}
\begin{IEEEbiography}[{\includegraphics[width=0.9in,clip,keepaspectratio]{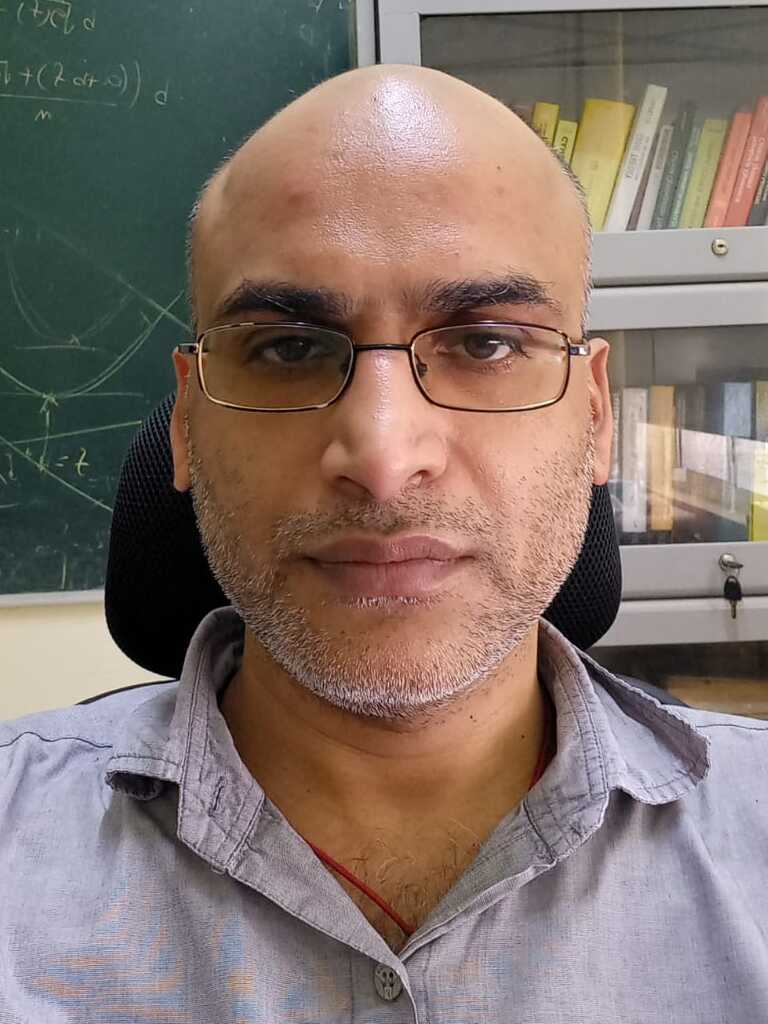}}]
	{Puduru Viswanadha Reddy} received the Ph.D. degree in operations research
	from Tilburg University, Tilburg, The Netherlands, in 2011. He is currently an Associate Professor with the Department of Electrical Engineering, IIT Madras, Chennai, India. His current research interests are 	in game theory and in the control of multi-agent systems.
\end{IEEEbiography} 
\begin{IEEEbiography}[{\includegraphics[width=0.9in,clip,keepaspectratio]{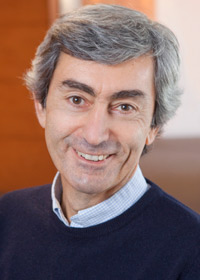}}] 
{Georges Zaccour} received the M.Sc. in international business and the Ph.D. in management science from HEC Montr\'{e}al, Montr\'{e}al, Canada, and a license in mathematics and economics from Universit\'{e} Paris-Dauphine, Paris, France. He holds the Chair in Game Theory and Management and is a Full Professor of Management Science at HEC Montréal. He served as Director of GERAD, an inter-university research center. His research interests include dynamic games, optimal control, and operations research, with applications to environmental management and supply chains. He has published three books, over 200 papers, and co-edited 13 volumes. Dr. Zaccour is a Fellow of the Royal Society of Canada and Editor-in-Chief of Dynamic Games and Applications. He served as President of the International Society of Dynamic Games from 2002 to 2006.
\end{IEEEbiography}
\end{document}